\documentclass[a4paper,10pt]{article}
\usepackage{geometry}                
\usepackage[active]{srcltx}
\usepackage{hyperref}
\usepackage[normalem]{ulem}

\usepackage{graphicx}
\usepackage{amssymb}
\usepackage{epstopdf}
\usepackage{amsmath}
\usepackage{amsthm}
\usepackage{amssymb}
\usepackage{latexsym}
\usepackage{mathtools}
\usepackage{mathrsfs}
\usepackage{graphics}
\usepackage{latexsym}
\usepackage{psfrag}
\usepackage{import}
\usepackage{verbatim}
\usepackage{enumerate}
\usepackage{enumitem}
\usepackage{graphicx}
\usepackage{bigints}
\usepackage[usenames]{color}

\newcommand{\dom}{{\rm{Dom\,}}}
\newcommand{\R}{\mathbb{R}}
\newcommand{\N}{\mathbb{N}}
\newcommand{\NN}{\mathcal{N}}

\newcommand{\supp}{\text{\rm supp}}
\newcommand{\loc}{\text{\rm loc}}

\newcommand{\Lip}{\mathrm{Lip}}

\newcommand{\diam}{{\rm{diam\,}}}

\newcommand{\ve}{\varepsilon}

\newcommand{\erre}{\mathbb{R}}

\newcommand{\T}{\mathcal{T}}

\renewcommand{\L}{\mathcal{L}}

\newcommand{\RCD}{\mathsf{RCD}}

\newcommand{\CD}{\mathsf{CD}}

\newcommand{\Geo}{{\rm Geo}}

\newcommand{\MCP}{\mathsf{MCP}}

\newcommand{\OptGeo}{{\rm OptGeo}}
\newcommand{\Ent}{{\rm Ent}}

\newcommand{\Ric}{{\rm Ric}}
\newcommand{\norm}[1]{\left\Vert#1\right\Vert}

\newcommand{\abs}[1]{\left\vert#1\right\vert}

\newcommand{\Real}{\mathbb{R}}

\renewcommand{\L}{\mathcal{L}}
\newcommand{\m}{\mathfrak{m}}

\newcommand{\vol}{\mathrm{Vol}}

\renewcommand{\P}{\mathbb P}

\renewcommand{\P}{\mathcal{P}}

\renewcommand{\H}{\mathcal{H}}

\newcommand{\mm}{\mathfrak m}
\newcommand{\qq}{\mathfrak q}

\newcommand{\ee}{{\rm e}}
\newcommand{\QQ}{\mathfrak Q}

\newcommand{\sfd}{\mathsf d}

\newcommand{\Ch}{\mathsf{Ch}}
\newcommand{\Per}{\mathsf{Per}}
\newcommand{\Opt}{\mathrm{OptGeo}}

\newcommand{\LIP}{\mathrm{LIP}}

\newcommand{\bd}{{\mathbf\Delta}}

\theoremstyle{plain}
\newtheorem{lemma}{Lemma}[section]
\newtheorem{theorem}[lemma]{Theorem}

\newtheorem{proposition}[lemma]{Proposition}
\newtheorem{corollary}[lemma]{Corollary}

\newtheorem*{theorem*}{Theorem}
\newtheorem*{maintheorem*}{Main Theorem}

\theoremstyle{definition}

\newtheorem{definition}[lemma]{Definition}
\newtheorem*{definition*}{Definition}
\newtheorem*{remark*}{Remark}
\newtheorem{remark}[lemma]{Remark}

\numberwithin{equation}{section}

\title{Indeterminacy estimates \\ and the size of nodal sets 
in singular spaces}

\begin{document}

\author{
Fabio Cavalletti\thanks{Mathematics Area, SISSA, Trieste (Italy) {\tt cavallet@sissa.it, sfarinel@sissa.it}}\, and 	
Sara Farinelli$^{*}$
}
	
\maketitle

\begin{abstract}
We obtain the sharp version 
of the uncertainty principle recently introduced in \cite{SagivSteinerberger}, and improved by \cite{CMO20}, 
relating the size of the zero set of a continuous function having zero mean 
and the optimal transport cost between the mass of the positive part and the negative one.
The result is actually valid for the wide family of metric measure spaces verifying 
a synthetic lower bound on the Ricci curvature, namely the $\MCP(K,N)$ or $\CD(K,N)$ condition, thus also extending the scope beyond the smooth setting of Riemannian manifolds.

Applying the uncertainty principle to eigenfunctions of the Laplacian in possibly non-smooth spaces,
 we obtain new lower bounds on the size of their nodal sets in terms of the eigenvalues. 
Those cases where the Laplacian is possibly non-linear are also covered 
and applications to linear combinations of eigenfunctions of the Laplacian are derived. 
To the best of our knowledge, no previous results were known for non-smooth spaces.  
\end{abstract}

\tableofcontents

\section{Introduction}

This paper is motivated by the recent emerging interest 
on uncertainty estimates and their applications to the behaviour 
of solutions of certain elliptic equations.

To be more precise: 
given a continuous function 
$f : \Omega \subset \R^{n} \to \R$ with zero mean 
$\int_{\Omega} f = 0$, with $\Omega$ compact, 
it is natural to interpret $f^{+}$, the positive part of $f$, 
and  $f^{-}$, the negative part of $f$, 
as two distributions of mass one can compare evaluating their Wasserstein distance $W_{1}$ (even though they do not have total mass 1).
Then, if it is cheap to transport $f^{+}dx$ to $f^{-}dx$ (meaning their 
 Wasserstein distance is small), 
necessarily most of the mass of $f^{+}$ must 
be close  to most of the mass of $f^{-}$. 
Continuity of $f$ implies then that necessarily  
the nodal set $\{x \in \Omega \colon f(x) = 0\}$ has to be large. 
Uncertainty estimates will quantify this relation.

This question was firstly investigated 
by Steinerberger in dimension 2 \cite{Steinerberger18}  and 
later in any dimension by Sagiv and Steinerberger \cite{SagivSteinerberger} proving that 
 any continuous function $f: [0,1]^{n} \to \R$ having zero mean
satisfies the following inequality
\begin{equation}\label{E:Steinerberger}
W_{1}(f^+\,dx,f^-\,dx)\cdot \H^{n-1}\left(\{ x\in X\colon f(x)=0\}\right)
\geq C \left(\frac{\norm{f}_{L^{1}}}{\|f \|_{L^{\infty}}}\right)^{4 - \frac{1}{n}}\| f \|_{L^{1}}.
\end{equation}
The constant $C$ depends only on $n$ and 
$\H^{n-1}$ denotes the Hausdorff measure of dimension $n-1$. 
Subsequently, an improvement on  \eqref{E:Steinerberger} 
has been obtained in \cite{CMO20} by Carroll, Massaneda and  Ortega-Cerd\`a showing the validity of 
\eqref{E:Steinerberger} with the better exponent $2-1/n$ and 
extending the range of applicability to continuous functions 
defined on any smooth and compact Riemannian manifold.

\smallskip
The first main result of this note is the following \emph{sharp} (in the exponent)
uncertainty principle valid for real valued, continuous or Sobolev functions defined over a metric measure spaces (m.m.s. for short) 
$(X,\sfd,\mm)$ verifying synthetic (meaning not requiring any smoothness assumption on $X$) 
lower bound on the Ricci curvature 
called Curvature-Dimension condition and denoted by $\CD(K,N)$, with $K$ mimiking the lower bound on 
the Ricci curvature and $N$ the upper bound on the dimension.
With the terminology metric measure space we intend 
a complete and separable metric space 
$(X,\sfd)$ endowed with a non-negative Radon measure $\mm$
(all the other terminology and notations will be introduced in Section \ref{S:preliminaries}).

\begin{theorem}[Sharp indeterminacy estimate]\label{T:main1}
Let $K,N \in \R$ with $N > 1$.
Let $(X,\sfd,\mm)$ be an essentially non-branching metric measure 
space verifying $\CD(K,N)$. 
Let $f \in L^{1}(X,\mm)$ be a continuous function, or alternatively $f \in W^{1,2}(X,\sfd,\mm)$, 
such that $\int_{X} f\, \mm = 0$ and assume the existence of $x_{0} \in X$ such that  $\int_{X} | f(x) |\,  \sfd(x,x_{0})\, \mm(dx)< \infty$. 

Then the following indeterminacy estimate is valid:
\begin{equation}\label{E:indeterminacyCDintro}
W_{1}(f^+\m,f^-\m)\cdot \Per\left(\{ x\in X\colon f(x)>0\}\right)
\geq \frac{\norm{f}_{L^1(X,\mm)}}{\|f \|_{L^{\infty}(X,\mm)}} 
\frac{\norm{f}_{L^1(X,\mm)}}{8C_{K,D}},
\end{equation}
where $D = \diam(X)$ and
$$
C_{K,D}: = 
\begin{cases} 
1 & K \geq 0, \\ e^{-K D^{2}/2} & K < 0.
\end{cases}
$$
\end{theorem}

The essentially non-branching assumption in Theorem \ref{T:main1}
is to prevent branch-like behaviour of geodesics and it is trivially satisfied by Riemannian manifolds and verified by the more regular class of $\RCD$ spaces. 
The notation $\Per(A)$ is used to denote the Perimeter of the set $A$ (see Section \ref{S:preliminaries} for its definition in this abstract setting). 
In the smooth setting, i.e. an $n$-dimensional Riemannian manifold endowed with the volume measure, it coincides thanks to  De Giorgi's Theorem with the $\H^{n-1}$-measure 
of the reduced boundary of $A$ (same result has been recently extended to the setting of non-collapsed $\RCD(K,N)$ spaces, see \cite{AmbrosioBrueSemola}).

We will now list few detailed comments on Theorem \ref{T:main1}. 

\smallskip
\noindent
\emph{Setting and Sharpness:}  
\begin{itemize}
\item[-] We improve on previous results
by including possibly non-smooth spaces, 
i.e. those spaces verifying the synthetic lower bound on the Ricci curvature, $\CD(K,N)$ condition, see 
Section \ref{subsec:cdbounds} for the precise definition and 
for a list of class of spaces falling within this theory.  
Here we mention that given a complete Riemannian manifold 
$(M,g)$ one can naturally consider the m.m.s. $(M,\sfd_{g}, \vol_{g})$ where 
where $\sfd_{g}$ is the geodesic distance and $\vol_{g}$ the volume measure
both induced by the metric $g$. 
Then $(M,\sfd_{g}, \vol_{g})$ verifies $\CD(K,N)$ 
if and only if $\Ric_{g} \geq K g$ and $\dim(M) \leq N$.
In particular \emph{any compact smooth weighted} (meaning with 
$\mm = e^{-V}\, \vol_{g}$ with $V$ smooth) \emph{Riemannian manifold} is included in our setting.  
Thanks to the well-known stability property of the $\CD(K,N)$ conditions 
in the measured-Gromov-Hausdorff sense, 
Theorem \ref{T:main1} applies also to any possible limit space 
of sequences of manifolds having $\Ric$ bounded from below uniformly 
and dimension bounded from above uniformly.

\item[-] The estimate does not require the space $X$ to be compact nor 
the reference measure $\mm$ to be finite, at least when $K \geq 0$. If $K <0$, then to have a meaningful estimate necessarily the diameter of $X$ has to be finite.  

\item[-] Inequality \eqref{E:indeterminacyCDintro} does not depend on the dimension. In particular the same statement is valid 
for m.m.s. satisfying $\CD(K,\infty)$ (i.e. no synthetic upper bound on the dimension) for which a localization paradigm is at disposal. In particular if $(X,\sfd,\mm)$ satisfies 
$\CD(K,\infty)$ and 
the weaker $\MCP(K',N')$ with some other $K',N'$, then 
\eqref{E:indeterminacyCDintro} is still valid. 
Referring to Theorem \ref{T:multidim} for the precise statement, 
here we underline that for any continuous function $f : \R^{n} \to \Real$ 
having zero mean and satisfying the growth assumptions 
with respect to $e^{-V} dx$ for some smooth convex function $V$, then 
\eqref{E:indeterminacyCDintro} holds true.

\item[-] As pointed out 
in \cite{CMO20} by Carroll, Massaneda and  Ortega-Cerd\`a, 
their version of \eqref{E:Steinerberger} cannot be improved 
by lowering the exponent $2-1/n$ below $1$; 
exponent 1 was known to be reached only in dimension 2. 
Hence the exponent 1 in \eqref{E:indeterminacyCDintro} is \emph{sharp}.

\item[-] Theorem \ref{T:main1} will be also valid for spaces 
verifying another synthetic curvature notion called  measure-contraction property $\MCP(K,N)$ (see Section \ref{subsec:cdbounds}). 
A long list of subRiemanninan spaces, including the Heinseberg group,  verifies this latter condition while failing $\CD(K,N)$.
In this framework the constant in the inequality \eqref{E:indeterminacyCDintro} will depend on the dimension.  
See Theorem \ref{T:multidimMCP} for the precise statement.
\end{itemize}

Our approach to obtain to Theorem \ref{T:main1} will be via a dimensional-reduction argument. 
In particular, the $L^{1}$-optimal transport problem between the positive and the 
negative part of $f$ gives, as a byproduct, a foliation of the ambient space $X$ into 
a family of geodesics obtained by considering the integral curves of the gradient of a 
Kantorovich potential $u$, i.e. a solution of the dual problem. 
This non-smooth foliation has few pleasant properties that are summarised in Theorem \ref{T:localization} (see Section \ref{S:preliminaries} for all preliminaries).
Here we mention that the integral of the function $f$ along almost every geodesics of the foliation is still zero, where the integral is with respect to the corresponding marginal measure. 
Also the curvature properties of the space are inherited by the one-dimensional ``weighted'' 
geodesic in a suitable sense. These two properties permit to reduce the 
proof of Theorem  \ref{T:main1} to a one-dimensional analysis (Section \ref{S:onedimensional}) and, in turn, to obtain the sharpness.

\subsection{Applications: Nodal set of eigenfunctions in singular spaces}
Our main application of Theorem \ref{T:main1}  will be a lower bound on the size of nodal sets for eigenfunctions of the Laplacian
(and linear combination of them) in possibly singular spaces verifying synthetic Ricci curvature bounds.

The whole list of topics related to the geometry 
of Laplace eigenfunctions (for instance the Courant nodal domain theorem or the quasi-symmetry conjecture) goes beyond the scope of this short introduction. 
However, to put the problem into perspective,
we will now recall  the long series of contributions to Yau's conjecture and 
the solution to it.  \\
Yau conjectured in \cite{Yau} that for any $n$-dimensional 
$C^{\infty}$-smooth closed Riemannian manifold $M$, hence without boundary and compact, any Laplace eigenfunction 
$$
-\Delta f_{\lambda} = \lambda f_{\lambda}
$$
satisfies 
$$
c\sqrt{\lambda} \leq \H^{n-1}(\{ f_{\lambda} = 0 \}) \leq C \sqrt{\lambda},
$$
with $c,C$ depend solely on $M$ and not on $\lambda$.

First Br\"uning \cite{Bruning} proved the validity 
of the lower bound for $n = 2$.
Then Donnelly and Fefferman in 1988 \cite{DonnellyFefferman88}
established Yau's conjecture in the case of real analytic metrics 
(for instance spherical harmonics). In the case 
of smooth manifold Nadirashvili in 1988 \cite{Nadirashvili88}
proved for  $n = 2$  that 
$\H^{1}(\{ f_{\lambda} = 0 \}) \leq C \lambda\log \lambda$
and later improved  \cite{DonnellyFefferman90,Dong92}
to $\H^{1}(\{ f_{\lambda} = 0 \}) \leq C \lambda^{3/4}$.
For general $n > 2$, Hardt and Simon \cite{HardtSimon} 
obtained the non-polynomial bound 
$\H^{n-1}(\{ f_{\lambda} = 0 \}) \leq C \lambda^{C\sqrt{\lambda}}$.

Few years later the lower bound has been improved to 
\begin{equation}\label{E:ColdingMinicozzi}
\H^{n-1}(\{ f_{\lambda} = 0 \}) \geq c \lambda^{\frac{3-n}{4}},
\end{equation}
in independent contributions by Colding and Minicozzi \cite{ColdingMinicozzi}, 
Sogge and Zelditch \cite{SoggeZelditch2011,SoggeZelditch2012} and by 
Steinerberger  \cite{Steinerberger14}.
Finally, a breakthrough has been obtained by Logunov in 2018, 
proving, in the smooth case and for any $n\in \N$, 
a polynomial upper bound 
\cite{Logunov1} and the lower bound \cite{Logunov2} in Yau's conjecture. For an overview on all these result we refer to 
\cite{LogunovMalinnikova}.

\smallskip
To the best of our knowledge there are no results 
on the size of  nodal sets of eigenfunctions of the Laplacian
whenever a singularity on the ambient manifold is allowed.
Following  Steinerberger \cite{Steinerberger18},
upper bounds on the $W_{1}$-distance between the positive and the negative parts of a Laplace eigenfunctions will yield
 lower bounds on the size of their nodal sets. 
This indeed is the content of the following results.
%
The first one will be for spaces verifying $\CD(K,N)$;
at this level of generality the Laplacian may not be even 
a linear operator (see Section \ref{Ss:Laplacian}).

\begin{theorem}[Nodal sets on $\CD$-spaces]\label{T:main3}
Let $K,N \in \R$ with $N > 1$.
Let $(X,\sfd,\mm)$ be an essentially non-branching  m.m.s. 
 verifying  $\CD(K,N)$ and such that $\mm(X) < \infty$. 
Let $f_{\lambda}$ be an eigenfunction of the Laplacian of eigenvalue 
$\lambda> 0$ (see Definition \ref{D:L2Laplace}) 
and assume moreover  the existence of $x_{0} \in X$ such that $\int_{X} | f_{\lambda}(x) |\,  \sfd(x,x_{0})\, \mm(dx)< \infty$.  

Then the following estimate on the size of the its nodal set holds true:
\begin{equation*}
\Per(\{x\in X : f_{\lambda}(x)>0 \})\geq   
\frac{\sqrt{\lambda}}{8C_{K,D}\sqrt{\mm(X)}}  
\cdot \frac{\|f_{\lambda}\|^2_{L^1(X,\mm)}}{\|f_{\lambda}\|_{L^{2}(X,\mm)}\|f_{\lambda}\|_{L^{\infty}(X,\mm)}},
\end{equation*}
where $D = \diam(X)$ and $C_{K,D}$ is the same of Theorem \ref{T:main1}.
\end{theorem}

As before, Theorem \ref{T:main3}  is actually valid 
also in other frameworks: for spaces verifying $\CD(K,\infty)$ 
and $\MCP(K',N')$ (Theorem \ref{T:nodalCD}) or  for spaces verifying
$\MCP(K,N)$ (Theorem \ref{T:nodalMCP}) with dimension dependent constant. 

If in addition we assume the Laplacian to be linear (more precisely the Sobolev space $W^{1,2}(X,\sfd,\mm)$ to be an Hilbert space), 
then more techniques from the classical setting come into play (for instance 
contraction estimates for the heat flow) 
permitting to obtain more refined results.

\begin{theorem}[Nodal sets on $\RCD$ spaces I]\label{T:main4} 
Let $K,N \in \R$ with $N > 1$.
Let $(X,\sfd,\mm)$ be a m.m.s. satisfying $\RCD(K,N)$,  
and such that $\diam(X) = D <\infty$.
Let $f_{\lambda}$ be an eigenfunction of the Laplacian of eigenvalue 
$\lambda> 2$.
Then the following estimate is valid:
\begin{equation}\label{E:NodalCDintro1}
\Per\left(\{ x\in X\colon f_{\lambda}(x)>0\}\right)
\geq
\frac{1}{\bar C_{K,D,N}}
\sqrt{ \frac{\lambda} {\log \lambda} }
\cdot 
\frac{\norm{f_{\lambda}}_{L^1(X,\mm)}}{ \|f_{\lambda} \|_{L^{\infty}(X,\mm)}},
\end{equation}
where $\bar C_{K,D,N}$ grows linearly in $D$ if $K \geq 0$ and exponentially if $K< 0$ and grows with power $1/2$ in $N$.
\end{theorem}

The estimate \eqref{E:NodalCDintro1} follows 
directly from the following estimate 
\begin{equation}\label{E:upperbound}
 W_1(f^+_{\lambda}\mm,f^-_{\lambda}\mm)
\leq C(K,N,D) 
\sqrt{\frac{\log \lambda }{\lambda}}\|f_{\lambda}\|_{L^1},
\end{equation}
(see Proposition \ref{P:upperboundW1}),
already obtained in the smooth setting by Steinerberger \cite{Steinerberger17}, and recently improved in \cite{CMO20}
removing the $\log \lambda$ term but with an approach that seems confined to the smooth setting.

Finally, we notice that using already available  $L^{\infty}$ estimates for Laplace eigenfunctions one can obtain an explicit lower bound on the size of the nodal set of an eigenfunction.

\begin{theorem}[Nodal sets on $\RCD$ spaces II]\label{T:main5} 
Let $K,N \in \R$ with $N > 1$.
Let $(X,\sfd,\mm)$ be a m.m.s. verifying $\RCD(K,N)$,
and with $\diam(X) = D <\infty$; finally pose $\mm(X) = 1$.
Let $f_{\lambda}$ be an eigenfunction of the Laplacian of eigenvalue 
$\lambda> \max{\lbrace 2, D^{-1}\rbrace}$.
Then the following estimate is valid:
\begin{equation}\label{E:NodalCDintro2}
\Per\left(\{ x\in X\colon f_{\lambda}(x)>0\}\right)
\geq
\frac{1}{\bar C_{K,D,N}}
\frac{1} {\sqrt{\log \lambda} } \lambda^{\frac{1-N}{2}},
\end{equation}
where $\bar C_{K,D,N}$ grows linearly in $D$ if $K \geq 0$ and exponentially if $K< 0$ and grows with power $1/2$ in $N$. 
\end{theorem}

To conclude we mention that if the Laplacian is linear, 
like for instance in the $\RCD$ setting, 
Theorem \ref{T:main3} and
Theorem \ref{T:main4} extend also to linear combinations 
of eigenfunction giving a non-smooth analogue of Sturm-Hurwitz' Theorem, see \cite{BerardHleffer20}; see Section \ref{Ss:infinitesimalHilbertian} for 
all the results.

\subsection{Outlook}\label{Ss:outlook}

Consider a compact, smooth, $N$-dimensional Riemannian manifold $M$ endowed with 
the geodesic distance $\sfd_{g}$ and the volume measure $\vol_{g}$. For this class of spaces 
the improved version of \eqref{E:upperbound} obtained in \cite{CMO20} holds true.
Hence our Theorem \ref{T:main1} gives the following inequality
$$
\H^{N-1}\left(\{ x\in X\colon f_{\lambda}(x)=0\}\right)
\geq
\frac{1}{\bar C_{K,D,N}}
\sqrt{\lambda} 
\cdot 
\frac{\norm{f_{\lambda}}_{L^1(M)}}
{ \|f_{\lambda} \|_{L^{\infty}(M)}}, 
$$
One can then use the inequality 
$\|f_{\lambda}\|_{\infty} 
\leq \lambda^{\frac{N-1}{4}}\| f_{\lambda} \|_{1}$ by Sogge and Zelditch \cite{SoggeZelditch2012} (which is known to be sharp on spherical harmonics) and obtain 
$$
\H^{N-1}\left(\{ x\in X\colon f_{\lambda}(x)=0\}\right)
\geq 
\lambda^{\frac{3-N}{4}},
$$
reproving the estimate \eqref{E:ColdingMinicozzi} 
by Colding and Minicozzi \cite{ColdingMinicozzi}, 
Sogge and Zelditch \cite{SoggeZelditch2011,SoggeZelditch2012} and by 
Steinerberger  \cite{Steinerberger14}.
It is therefore plausible to expect \eqref{E:ColdingMinicozzi} (or its counterpart with the Perimeter) to holds true 
also for compact $\RCD$-spaces, provided the the validity of the following two inequalities
is established
$$
W_1(f^+_{\lambda}\mm,f^-_{\lambda}\mm)
\leq C(K,N,D) 
\frac{1}{\sqrt{\lambda}}\|f_{\lambda}\|_{L^1},
\qquad \|f_{\lambda}\|_{L^\infty} 
\leq \lambda^{\frac{N-1}{4}}\| f_{\lambda} \|_{L^1}.
$$
that are left for a future investigation. Similar investigation will be also carried out for the quasi-symmetry property of eigenfunction in the non-smooth setting.

\section{Preliminaries}
\label{S:preliminaries}

In what follows, $(X,\sfd,\mm)$  
will be a complete and separable metric measure space that is 
$(X,\sfd)$ is a complete and separable metric space 
 and $\mm$ is a non-negative Radon measure on $X$. 
Also, throughout the note, the various curvature conditions
we will assume will imply $X$ to be proper (bounded and closed sets are compact). 
In various situation this will simplify the presentation 
(see Section \ref{Ss:Laplacian}).


\subsection{Synthetic Curvature conditions}
\label{subsec:cdbounds}

We briefly recall the main definitions of curvature bounds for metric measure spaces that we will use throughout the paper 
referring for more details to the original papers 
\cite{lottvillani:metric, sturm:I, sturm:II}.

In the following $\mathcal{P}(X)$ is the space of Borel probability measures on $X$ and, for $p\geq 1$, $\mathcal{P}_{p}(X)$ is
the space of Borel probability measures with finite $p$-moment.

The \emph{p-Wasserstein distance} 
$W_p$ on $\mathcal{P}_p(X)$ is defined for any $\mu_0,\mu_1 \in \mathcal{P}_p(X)$
as follows:
\begin{equation}\label{eq:wasser_dist}
W_p(\mu_0,\mu_1)^{p} : = 
\inf_{\pi \in\Pi(\mu_{0},\mu_{1}) }  \int_{X \times X} \sfd^p (x,y)\,\pi(dxdy), \quad 
\end{equation}
where 
$$
\Pi(\mu_{0},\mu_{1}): =
\left\{ \pi \in \mathcal{P} (X \times X) \colon 
P^{(1)}_\sharp \pi = \mu_0, P^{(2)}_\sharp \pi = \mu_1 \right\}
$$
is the set of admissible transport plans between $\mu_{0}$ and $\mu_{1}$ and 
$P^{(i)}$ is the projection on the $i$-th component, for $i = 1,2$.
We will only considering in this note $W_{1}$ and $W_{2}$.
$\Geo(X)$ denotes the space of constant speed geodesics on $X$:
$$
\Geo(X): = \left\{\gamma \in C([0,1],X) \colon 
			\text{$\sfd(\gamma(s),\gamma(t))=|s-t|\sfd(\gamma(0),\gamma(1))$ for any $s,t\in[0,1]$}\right\}.
$$
For any $t\in[0,1]$, the evaluation map $\ee_t$ is defined on 
$\Geo(X)$ by
$ \ee_t(\gamma) : =  \gamma(t)$.
For any pair of measures $\mu_0$, $\mu_1$ in 
$\mathcal{P}_2(X)$, 
the set of dynamical optimal plans is defined by
$$
\OptGeo(\mu_0,\mu_1)
		: =	\left\{
		\nu \in \mathcal{P}(\Geo(X)) \colon
(\ee_0,\ee_1)_\sharp \nu \quad \text{realizes the minimum in \eqref{eq:wasser_dist}}\right\}.
$$
\begin{definition}[Essentially non-branching]  
A subset $G \subset \Geo(X, \sfd)$ 
of geodesics is called non-branching if for any $\gamma^{1}, \gamma^{2} \in G$ 
the following holds: 
$$
\exists \, t \in (0,1) \colon \gamma^{1}_{s} = \gamma^{2}_{s} \quad \forall \ s \in [0,t]
\Longrightarrow \ \gamma^{1}_{2} = \gamma^{2}_{s} \quad \forall \ s\in [0,1].
$$
$(X, \sfd)$ is called non-branching if $\Geo(X,\sfd)$ is non-branching. 
$(X, \sfd, \mm)$ is called essentially non-branching if for any 
$\mu_{0},\mu_{1} \ll \mm$ with $\mu_{0},\mu_{1} \in \mathcal{P}_{2}(X)$ 
any $\nu \in \OptGeo(\mu_{0},\mu_{1})$
is concentrated on a Borel non-branching subset $G\subset \Geo(X, \sfd)$.
\end{definition}
The above definition was introduced in \cite{RS2014} by Rajala and Sturm.
The restriction to essentially non-branching spaces is natural and facilitates avoiding pathological cases. One example is the failure of the local-to-global property for 
a general $\CD(K,N)$ in \cite{R2016},
property that has been recently verified in \cite{CM16} under the assumption of essentially non-branching (and finite 
$\mm$).

\smallskip
Given $K \in \Real$ and $\NN \in (0,\infty]$, define:
\begin{equation}\label{E:diameter}
D_{K,\NN} := \begin{cases}  \frac{\pi}{\sqrt{K/\NN}}  & K > 0 \;,\; \NN < \infty \\ +\infty & \text{otherwise} \end{cases} .
\end{equation}
In addition, given $t \in [0,1]$ and $0 < \theta < D_{K,\NN}$, define:
$$
\sigma^{(t)}_{K,\NN}(\theta) := \frac{\sin(t \theta \sqrt{\frac{K}{\NN}})}{\sin(\theta \sqrt{\frac{K}{\NN}})} = 
\begin{cases}   
\frac{\sin(t \theta \sqrt{\frac{K}{\NN}})}{\sin(\theta \sqrt{\frac{K}{\NN}})}  & K > 0 \;,\; \NN < \infty \\
t & K = 0 \text{ or } \NN = \infty \\
 \frac{\sinh(t \theta \sqrt{\frac{-K}{\NN}})}{\sinh(\theta \sqrt{\frac{-K}{\NN}})} & K < 0 \;,\; \NN < \infty 
\end{cases} ,
$$
and set $\sigma^{(t)}_{K,\NN}(0) = t$ and $\sigma^{(t)}_{K,\NN}(\theta) = +\infty$ for $\theta \geq D_{K,\NN}$. 
Given $K \in \Real$ and $N \in (1,\infty]$, 
the \emph{distortion coefficients} are defined as:
$$
\tau_{K,N}^{(t)}(\theta) := t^{\frac{1}{N}} \sigma_{K,N-1}^{(t)}(\theta)^{1 - \frac{1}{N}} .
$$
When $N=1$, set $\tau^{(t)}_{K,1}(\theta) = t$ if $K \leq 0$ and $\tau^{(t)}_{K,1}(\theta) = +\infty$ if $K > 0$.

\smallskip

The \emph{R\'enyi entropy} functional 
$\mathcal{E}:\mathcal{P} (X) \to [0,\infty]$ is defined as
$$
\mathcal{E}(\mu): = 
\int_{X} \rho^{1-\frac{1}{N}}\,\mm, \quad 
\text{where } \mu = \rho\mm+\mu^s \text{ and } 
\mu^s \perp \mm,
$$
and the {\em Boltzman entropy} 
$\Ent : \mathcal{P} (X) \to [0,\infty]$
defined by 
$$
\Ent(\mu): = 
\int_{X} \rho \log (\rho)\,\mm, \quad 
\text{if } \mu = \rho\mm, \text{ and } 
\Ent(\mu) : = \infty \text{ otherwise } 
$$

\begin{definition}[$\CD$ conditions]
$(X,\sfd,\mm)$ verifies the $\CD(K,N)$ (resp. $\CD(K,\infty)$) condition for some $K\in \R$, $N\in(1,\infty)$ 
if for any pair of probability measures  
$\mu_0,\mu_1\in\mathcal{P}(X)$ 
 with $\mu_0,\mu_1\ll \mm$ (and $\Ent(\mu_{i})<\infty$, $i = 0,1$), 
there exists $\nu \in \OptGeo(\mu_0,\mu_1)$ 
and an optimal plan $\pi\in \Pi(\mu_{0},\mu_{1})$ 
such that $\mu_t : = (\ee_t)_\sharp \nu \ll \mm$ 
and 
$$
\mathcal{E}_{N'}(\mu_t)\geq 
\int\left\{
\tau^{(1-t)}_{K,N'}(\sfd(x,y))\rho_0^{-\frac{1}{N'}}+
\tau^{(t)}_{K,N'}(\sfd(x,y))  \rho_1^{-\frac{1}{N'}} 
				\right\}  \pi (dxdy)
$$
for any $N'\geq N$, $t\in[0,1]$ (resp .
$$
\Ent(\mu_t)\leq 
(1-t)\Ent(\mu_{0}) + t\Ent(\mu_{1}) - \frac{K}{2}t(1-t) 
W_{2}(\mu_{0},\mu_{1})^{2}\text{).}
$$

\end{definition}

For our purposes we also need to introduce a weaker variant of $\CD$ 
called Measure-Contraction property, $\MCP(K,N)$ in short, 
introduced separately by Ohta \cite{Ohta1} and Sturm \cite{sturm:II} with
 two definitions that slightly differ in general metric spaces, but that  coincide on essentially non-branching spaces.

\begin{definition}[$\MCP(K,N)$] \label{D:Ohta1}
A m.m.s. $(X,\sfd,\mm)$ is said to satisfy $\MCP(K,N)$ if for any $o \in \supp(\mm)$ and  $\mu_0 \in \P_2(X,\sfd,\mm)$ of the form $\mu_0 = \frac{1}{\mm(A)} \mm\llcorner_{A}$ for some Borel set $A \subset X$ with $0 < \mm(A) < \infty$ (and with $A \subset B(o, \pi \sqrt{(N-1)/K})$ if $K>0$), there exists $\nu \in \Opt(\mu_0, \delta_{o} )$ such that:
\begin{equation} \label{eq:MCP-def}
\frac{1}{\mm(A)} \mm \geq (\ee_{t})_{\sharp} \big( \tau_{K,N}^{(1-t)}(\sfd(\gamma_{0},\gamma_{1}))^{N} \nu(d \gamma) \big) \;\;\; \forall t \in [0,1] .
\end{equation}
\end{definition}

If $(X,\sfd,\mm)$ is a m.m.s. verifying $\MCP(K,N)$, then $(\supp(\mm),\sfd)$  is Polish, proper and it is a geodesic space. 
With no loss in generality for our purposes we will assume that $X = \supp(\mm)$.

\smallskip

To conclude this part we include a list of notable examples of spaces fitting in the assumptions of our results. 
The class of essentially non branching  $\CD(K,N)$ 
spaces includes many remarkable family of spaces, among them:
\begin{itemize}
\item \emph{Measured Gromov Hausdorff limits of Riemannian $N$-dimensional manifolds  satisfying  ${\rm Ric}_g\ge Kg$ and more generally the class of $\RCD(K,N)$ spaces}.
Indeed measured Gromov Hausdorff limits of Riemannian $N$-manifolds satisfying  ${\rm Ric}_g\ge Kg$  are examples of $\RCD(K,N)$ spaces (see for instance \cite{EKS} and for the definition of $\RCD$ see Section \ref{Ss:Laplacian}) and, in particular, 
are essentially non-branching and $\CD(K,N)$ (see \cite{RS2014}).
\item \emph{Alexandrov spaces with curvature $\geq K$}.
Petrunin \cite{PLSV} proved that the lower curvature bound in the sense of comparison triangles is compatible with the optimal transport type lower bound on the  Ricci curvature given by Lott-Sturm-Villani.  Moreover  geodesics in Alexandrov spaces with curvature bounded below do not branch. It follows that Alexandrov spaces with curvature bounded from below by $K$ are non-branching  $\CD(K(N-1),N)$ spaces.

\item \emph{Finsler manifolds where the norm on the tangent spaces is strongly convex, and which satisfy lower Ricci curvature bounds.} More precisely we consider a $C^{\infty}$-manifold  $M$, endowed with a function $F:TM\to[0,\infty]$ such that $F|_{TM\setminus \{0\}}$ is $C^{\infty}$ and  for each $p \in M$ it holds that $F_p:=T_pM\to [0,\infty]$ is a  strongly-convex norm, i.e.
$$\qquad \quad \; g^p_{ij}(v):=\frac{\partial^2 (F_p^2)}{\partial v^i \partial v^j}(v) \quad \text{is a positive definite matrix at every } v \in T_pM\setminus\{0\}. $$
Under these conditions, it is known that one can write the  geodesic equations and geodesics do not branch: in other words these spaces are non-branching.
We also assume $(M,F)$ to be geodesically complete and endowed with a $C^{\infty}$ probability measure $\mm$ in a such a way that the associated m.m.s. $(X,F,\mm)$ satisfies the $\CD(K,N)$ condition. This class of spaces has been investigated by Ohta \cite{Ohta} who established the equivalence between the Curvature Dimension condition and a Finsler-version of Bakry-Emery $N$-Ricci tensor bounded from below.
\end{itemize}

While $\CD(K,N)$ implies the weaker $\MCP(K,N)$, the latter is 
capable to capture the behaviour of more general family of spaces. 
In particular, for a complete list of subRiemannian spaces verifying 
the $\MCP(K,N)$ (and not $\CD(K,N)$), we refer to the recent \cite{Milman}.

\subsection{Localization and one-dimensional densities}

One of the key tools of our approach to obtain a sharp indeterminacy estimate 
is the dimensional reduction argument furnished by 
localization theorem. 
In its various forms, the following theorem goes back to 
\cite{biacava:streconv} for the $\MCP$ case (with a slightly different presentation), while to \cite{CM1} for the $\CD(K,N)$ case with $\mm(X) <\infty$ and to 
\cite{CM18} for a general Radon measure. 
We refer to the aforementioned references for all the missing details.

\begin{theorem}\label{T:localization}
Let $(X,\sfd, \mm)$ be an essentially non-branching metric measure space with $\supp(\mm) = X$. 
Let $f : X \to \R$ be $\mm$-integrable such that $\int_{X} f\, \mm = 0$ and assume the existence of $x_{0} \in X$ such that $\int_{X} | f(x) |\,  \sfd(x,x_{0})\, \mm(dx)< \infty$. 

Assume also $(X,\sfd, \mm)$ verifies
$\CD(K,N)$ (resp. $\MCP(K,N)$) condition for some $K\in \R$ and $N\in [1,\infty)$.

\medskip

Then the space $X$ can be written as the disjoint union of two sets $Z$ and $\mathcal{T}$ with $\mathcal{T}$ admitting a partition 
$\{ X_{\alpha} \}_{\alpha \in Q}$ and a corresponding disintegration of $\mm\llcorner_{\mathcal{T}}$ such that: 
$$
\mm\llcorner_{\T} = \int_{Q} \mm_{\alpha} \, \qq(d\alpha),
$$
where $\qq$ is a Borel probability measure over $Q \subset X$ such that 
$\QQ_{\sharp}( \mm\llcorner_{\T} ) \ll \qq$, with $\QQ$ the quotient map associated to the partition
and the map 
$Q \ni \alpha \mapsto \mm_{\alpha} \in \mathcal{M}_{+}(X)$ satisfying the following properties:

\begin{itemize}
\item for any $\mm$-measurable set $B$, the map $\alpha \mapsto \mm_{\alpha}(B)$ is $\qq$-measurable;
\item 
for $\qq$-a.e. $\alpha \in Q$, $\mm_{\alpha}$ is concentrated on $\QQ^{-1}(\alpha) = X_{\alpha}$ (strong consistency);
\item For $\qq$-almost every $q \in Q$, it holds $\int_{X_{q}} f \, \mm_{q} = 0$ and $f = 0$ $\mm$-a.e. in $Z$.
\item For $\qq$-almost every $q \in Q$, the set $X_{q}$ is a geodesic (even more a transport ray) and 
the one dimensional m.m.s. $(X_{\alpha},\sfd, \mm_{\alpha})$ 
verifies $\CD(K,N)$ (resp. $\MCP(K,N)$).
\end{itemize}
Moreover, fixed any $\qq$ as above such that $\QQ_{\sharp}( \mm\llcorner_{\T} ) \ll \qq$, the disintegration is $\qq$-essentially unique.
\end{theorem}

\begin{remark}\label{R:realinterval}
Via the ray map $g$ associated to the transport set of $f^+\mm$ into $f^-\mm$ (see for instance \cite{biacava:streconv}),  
we have that for $\qq$-a.e. $q \in Q$ 
$$
\mm_{q} = g(q,\cdot)_\sharp  \left(  h_{q} \cdot \mathcal{L}^{1} \right),  
$$
for some function $h_{q} :  \dom (g(q,\cdot)) \subset \R \to [0,\infty)$ where $\dom (g(q,\cdot))$ is an interval $I_q\subset \Real$ and $(I_q,\,|\cdot|,\, h_{q} \cdot \mathcal{L}^{1})$ is isomorphic to $(X_{\alpha},\sfd, \mm_{\alpha})$; in particular 
it verifies $\CD(K,N)$ (resp. $\MCP(K,N)$).
\end{remark}
\medskip

We will therefore spend few lines on one-dimensional m.m.s. verifying curvature bounds.

\begin{definition}[$\CD(K,N)$ density] \label{def:CDKN-density}
Given $K,N \in \Real$ and $N \in (1,\infty)$, a non-negative function $h$ defined on an interval $I \subset \Real$ is called a $\CD(K,N)$ density on $I$,  if for all $x_0,x_1 \in I$ and $t \in [0,1]$:
$$
 h(t x_1 + (1-t) x_0)^{\frac{1}{N-1}} \geq  \sigma^{(t)}_{K,N-1}(\abs{x_1-x_0}) h(x_1)^{\frac{1}{N-1}} + \sigma^{(1-t)}_{K,N-1}(\abs{x_1-x_0}) h(x_0)^{\frac{1}{N-1}} ,
$$
(recalling the coefficients $\sigma$ from Section \ref{subsec:cdbounds}). 
The case $N = \infty$ request instead
$$
 \log h (t x_1 + (1-t) x_0) \geq t \log h(x_1) + (1-t) \log h(x_0) + \frac{K}{2} t (1-t) (x_1-x_0)^2, 
$$
obtained from the previous one
subtracting 1 from both sides, multiplying by $N-1$, and taking the limit as $N \rightarrow \infty$.
For completeness, we will say that $h$ is a $\CD(K,1)$ density on $I$ iff $K \leq 0$ and $h$ is constant on the interior of $I$. 
\end{definition}

\begin{definition}[$\MCP(K,N)$ density]
Given $K,N \in \Real$ and $N \in (1,\infty)$, a non-negative function $h$ defined on an interval $I \subset \Real$ is called a $\MCP(K,N)$ density on $I$ if for all $x_0,x_1 \in I$ and $t \in [0,1]$:
\begin{equation}\label{E:MCPdef}
 h(t x_1 + (1-t) x_0) \geq \sigma^{(1-t)}_{K,N-1}(\abs{x_1-x_0})^{N-1} h(x_0).
\end{equation}
\end{definition}

The link between one dimensional m.m.s. with curvature bounds and densities is 
contained in the next straightforward result.

\begin{theorem}
If $h$ is a $\CD(K,N)$ (resp. $\MCP(K,N)$) density on an interval $I \subset \R$ then the m.m.s. $(I,\abs{\cdot},h(t) dt)$ verifies $\CD(K,N)$ (resp. $\MCP(K,N)$). 

Conversely, if the m.m.s. $(\Real,\abs{\cdot},\mu)$ verifies $\CD(K,N)$ (resp. $\MCP(K,N)$) and $I = \supp(\mu)$ is not a point, then $\mu \ll \L^1$ and there exists a version of the density $h = d\mu / d\L^1$ which is a $\CD(K,N)$ (resp. $\MCP(K,N)$) density on $I$.\end{theorem}

The estimate \eqref{E:MCPdef} implies several known properties that we collect in what follows. 
To write them in a unified way we 
define for $\kappa \in \erre$ the function $s_{\kappa} : [0,+\infty) \to \erre$ (on $[0,\pi/\sqrt{\kappa})$ if $\kappa>0$) 
\begin{equation}\label{E:sk}
s_{\kappa}(\theta):= 
\begin{cases}
(1/\sqrt{\kappa})\sin(\sqrt{\kappa}\theta) & {\rm if}\ \kappa>0, \crcr
\theta & {\rm if}\ \kappa=0, \crcr
(1/\sqrt{-\kappa})\sinh(\sqrt{-\kappa}\theta) &{\rm if}\ \kappa<0.
\end{cases}
\end{equation}
For the moment we confine ourselves to the case $I = (a,b)$ with $a,b \in \R$; hence
\eqref{E:MCPdef} implies
\begin{equation}\label{E:MCPdef2}
\left( \frac{s_{K/(N-1)}(b - x_{1}  )}{s_{K/(N-1)}(b - x_{0}  )} \right)^{N-1} 
\leq \frac{h(x_{1} ) }{h (x_{0})} \leq 
\left( \frac{s_{K/(N-1)}( x_{1} -a  )}{s_{K/(N-1)}( x_{0} -a  )} \right)^{N-1}, 
\end{equation}
for $x_{0} \leq x_{1}$.
Hence denoting with $D = b - a$ the length of $I$,  for any $\ve >0$ it follows that
\begin{equation}\label{E:bounds}
\sup \left\{ \frac{h (x_{1})}{h (x_{0})} \colon x_{0},x_{1} \in [a+ \ve, b - \ve ]\right\} \leq C_{\ve},
\end{equation}
where $C_{\ve}$ only depends on $K,N$, provided $ 2 \ve \leq D \leq \frac{1}{\ve}$.
In particular, $\MCP(K,N)$ densities will be locally Lipschitz in the interior of their 
domain and continuous on its closure (see \cite{CM18} for details).

\smallskip
To conclude we present here a folklore result about localization 
paradigm in the setting of $\CD(K,\infty)$ spaces.
So far Theorem \ref{T:localization} is not known for
a general $\CD(K,\infty)$ spaces, the missing ingredient being 
good behaviour of $W_{2}$-geodesics. Additionally assuming 
the space to satisfy $\MCP(K',N')$ for some $K',N' \in \R$ 
(with possibly $K'$ different from $K$) excludes all the technical 
issues and the proof of the following localization result 
just follows as the one of Theorem \ref{T:localization}. 

\begin{theorem}\label{T:localizationinfty}
Let $(X,\sfd, \mm)$ be an essentially non-branching metric measure space with $\supp(\mm) = X$. 
Let $f : X \to \R$ be $\mm$-integrable such that $\int_{X} f\, \mm = 0$ and assume the existence of $x_{0} \in X$ such that $\int_{X} | f(x) |\,  \sfd(x,x_{0})\, \mm(dx)< \infty$. 

Assume also $(X,\sfd, \mm)$ verifies
$\CD(K,\infty)$ and $\MCP(K',N')$ conditions for some $K,K'\in \R$ and $N'\in [1,\infty)$.

\medskip

Then the space $X$ can be written as the disjoint union of two sets $Z$ and $\mathcal{T}$ with $\mathcal{T}$ admitting a partition 
$\{ X_{\alpha} \}_{\alpha \in Q}$ and a corresponding disintegration of $\mm\llcorner_{\mathcal{T}}$ such that: 
$$
\mm\llcorner_{\T} = \int_{Q} \mm_{\alpha} \, \qq(d\alpha),
$$
where $\qq$ is a Borel probability measure over $Q \subset X$ such that 
$\QQ_{\sharp}( \mm\llcorner_{\T} ) \ll \qq$ and the map 
$Q \ni \alpha \mapsto \mm_{\alpha} \in \mathcal{M}_{+}(X)$ satisfies the following properties:

\begin{itemize}
\item for any $\mm$-measurable set $B$, the map $\alpha \mapsto \mm_{\alpha}(B)$ is $\qq$-measurable;
\item 
for $\qq$-a.e. $\alpha \in Q$, $\mm_{\alpha}$ is concentrated on $\QQ^{-1}(\alpha) = X_{\alpha}$ (strong consistency);
\item For $\qq$-almost every $q \in Q$, it holds $\int_{X_{q}} f \, \mm_{q} = 0$ and $f = 0$ $\mm$-a.e. in $Z$.
\item For $\qq$-almost every $q \in Q$, the set $X_{q}$ is a geodesic (even more a transport ray) and 
the one dimensional m.m.s. $(X_{\alpha},\sfd, \mm_{\alpha})$ 
verifies $\CD(K,\infty)$.
\end{itemize}
Moreover, fixed any $\qq$ as above such that $\QQ_{\sharp}( \mm\llcorner_{\T} ) \ll \qq$, the disintegration is $\qq$-essentially unique.
\end{theorem}

\subsection{Perimeters}
\label{subsec:perim} 
Given a metric measure space $(X,\sfd,\mm)$, one can introduce a notion 
of \emph{perimeter} which extends the classical one on $\R^n$. 
The following presentation follows \cite{Miranda} and the more recent 
\cite{AmbDiMarino}. 
We start by recalling the notion of slope (or local Lipschitz constant) of a real-valued function.

\begin{definition}[Slope]
Let $(X,\sfd)$ be a metric space and $u:X\to \R$ be a real valued function. We define the \emph{slope} of $f$ at the point $x\in X$ as
\begin{equation*}
\abs{\nabla u}(x): =  
	\begin{cases}
		\limsup_{y\to x} \frac{\abs{u(x)-u(y)}}{d(x,y)} 	&\text{if $x$ is not isolated}		\\
		0		&\text{otherwise}.
\end{cases}
\end{equation*}
\end{definition}

To fix notations, 
the space of Lipschitz maps on $(X,\sfd)$  
will be denoted by $\Lip (X)=\Lip(X,\sfd)$ 
while $\Lip_c (X)=\Lip_c(X,\sfd)$ 
will be the subspace of compactly supported Lipschitz maps.
If the function is locally Lipschitz in an open set A, i.e.  for every $x\in A$, the function is Lipschitz in a neighborhood of x, then we use the notation $\Lip_{loc}(A)$

\begin{definition}[Perimeter]
Let $E\in\mathcal{B}(X)$, where $\mathcal{B}(X)$ denotes the class of Borel sets of $(X,\sfd)$, and let 
$A\subset X$ be open. 
We define the perimeter of $E$ relative to $A$ as:
$$
\Per (E;A): = \inf \left\{\liminf_{n\to \infty} \int_A \abs{\nabla u_n} \mm \colon\, u_n\in \Lip_{loc}(A),\,\, u_n\to \chi_E \text{ in } L^1_{loc}(A,\mm)\right\},
$$
where $\abs{\nabla u}(x)$ is the slope of $u$ at the point $x$.
If $\Per (E;X)<\infty$, we say that $E$ is a \emph{set of finite perimeter}.
We denote $\Per(E;X)$ with $\Per(E)$.
\end{definition}

When $E$ is a fixed set of finite perimeter, the map 
$A\mapsto \Per (E;A)$ is the restriction to open sets of a finite Borel measure on $X$, 
defined as
$$
\Per(E;B): = \inf \left\{\Per(E;A) \colon \text{$A$ open}, A\supset B\right\}.
$$
For some recent progress on the extension of De Giorgi's rectifiability theorem (relating the perimeter and the Hausdorff measure of codimension 1)  
to the setting of non-collapsed $\RCD(K,N)$ spaces, 
we refer to \cite{AmbrosioBrueSemola} and references therein.

We observe the following fact.
\begin{lemma}\label{L:perineq}
Let $(X,\sfd,\mm)$ be a metric measure space, $E \subseteq X$ be a Borel set. 
Assume that we are given a strongly consistent disintegration of $\mm$ associated to a zero mean function  
as given in Theorem \ref{T:localizationinfty}: 
$$
\mm\llcorner_{\T} = \int_{Q}\mm_{\alpha}\,\qq(d\alpha),
$$
where $\qq$ is a Borel probability measure over $Q \subset X$ such that 
and $\mm_{\alpha} \in \mathcal{M}_{+}(X)$.
Then it holds 
\begin{equation*}
\Per(E)\geq \int_{Q}\Per_\alpha(E_\alpha)\,\qq(d\alpha),
\end{equation*}
where $E_\alpha=E\cap X_\alpha$ and $\Per_{\alpha}$ is the perimeter functional
in the space $(X_\alpha,\sfd,\mm_\alpha)$. 
\end{lemma}
	
\begin{proof}

Let $\lbrace f_n\rbrace_n\in \Lip_{\loc}(X)$ be a sequence of functions 
converging in $L^1(X,\mm)$ to $\chi_{E}$. Then, by disintegration 
\begin{equation*}
0=\lim_{n\to+\infty}\int_{X}|f_n(x)-\chi_E(x)|\,\mm(dx) 
= \lim_{n\to+\infty}\int_{Q}\int_{X_{\alpha}}|f_n(x)-\chi_E(x)|\,\mm_{\alpha}(dx)\,\qq(d\alpha),
\end{equation*}
so up to extracting a subsequence, that we call 
again  $\lbrace f_{n}\rbrace$, we have that for $\qq$-a.e. $q\in {Q}$
\begin{equation*}
\lim_{n\to +\infty}\int_{X_{\alpha}}|f_n(x)-\chi_E(x)|\,\mm_\alpha(dx)=0.
\end{equation*}
Recalling that each $\mm_\alpha$ is concentrated on $X_\alpha$ and denoting 
$E_\alpha:=E \cap X_\alpha$, we have that $f_{n}\llcorner_{X_\alpha}$ converges on $L^1(X_\alpha,\mm_\alpha)$ to $\chi_{E_\alpha}$ for $\qq$-a.e $\alpha\in Q$. 
We observe in addition that if $f_n$ is Lipschitz then 
$f_{n}\llcorner_{X_\alpha}$ is Lipschitz as well with a smaller local Lipschitz constant.
Hence, taken $\lbrace f_n\rbrace_n\in \Lip_{\loc}(X)$ a sequence 
of functions attaining in the limit $\Per(E)$, we have that 
\begin{align*}
\Per(E)&~= \liminf_{n\to\infty}\int_X |\nabla f_n|\, \mm
\geq \liminf_{n\to\infty}\int_Q\int_{X_\alpha} |\nabla f_n|\, \mm_\alpha\, \qq(d\alpha) \\
& ~\geq  \liminf_{n\to\infty}\int_Q\int_{X_\alpha} |\nabla f_n\llcorner_{X_{\alpha}}|\, \mm_\alpha\, \qq(d\alpha)
\geq \int_Q \liminf_{n\to\infty}\int_{X_\alpha} |\nabla f_n\llcorner_{X_{\alpha}}|\, \mm_\alpha\, \qq(d\alpha) \\
&~\geq \int_Q \Per_\alpha(E_\alpha)\, \mm_\alpha\, \qq(d\alpha),
\end{align*}
and the claim follows.
\end{proof}
	
%

We also include the following easy fact about the perimeter in the weighted one dimensional case. For a proof we refer to  \cite[Proposition 3.1]{CM16b} where there is an analogous statement for $\CD(K,N)$ densities. \\

\begin{lemma}\label{L:Perreal}
	Let 
	$\mm= h\mathcal{L}^1$ be a non-negative measure on $\Real$, with $h$ a $\CD(K,\infty)$ density on its support, which in particular is an interval. Let $E$ be an open set in $\supp(\mm)$. 
	Let $C_1,\dots,C_n$ be its connected components, with $n$ possibly $+\infty$. We consider the set $\cup_{k=0}^{n}\bar{C}_k$. We observe that  $\cup_{k=0}^{n}\bar{C}_k=\cup_{k=0}^{m}[a_k,b_k]$, with $[a_k,b_k]$ disjoint, with $m$ possibly $+\infty$, $a_k,b_k\in\Real\cup\lbrace\pm \infty\rbrace$. Then, setting  $B(E):=\cup _{k=0}^{m}\lbrace a_k, b_k\rbrace\setminus\lbrace \inf(\supp{( \mm )}), \sup(\supp(\mm)) \rbrace $, it holds
	$$
	\Per_{h}(E)= \sum_{ x\in B(E)}h(x)=\sum_{k=1}^{m}h(a_k)+h(b_k),
	$$
	where $\Per_{h}$ is the Perimeter functional in the space $(\supp(\mm),|\cdot|, h \L^{1})$.
\end{lemma}	

\medskip
\subsection{Laplacian, Heat Flow and $\RCD$}\label{Ss:Laplacian}
The main references for this part are 
\cite{Cheeger, AGS11a, AGS11b, Gigli12,EKS, AMS2013, CMi16} 
or  \cite{AmbrosioICM} for a survey on the subject.

We recall the definition of 
the Cheeger energy of an $L^2$ function, which will be used to define Sobolev spaces on metric measure spaces. 
We will be only concerned with the case $p = 2$.

Let $f\in L^p(X, \mm)$, 
the \emph{Cheeger energy} of $f$ is defined as
\begin{equation}\label{eq:cheeger_energy}
	\Ch(f): = \inf \left\{\liminf_{n\to \infty}\frac{1}{2}\int \abs{\nabla f_n}^2 \, \mm \colon 
f_n\in \Lip(X)\cap L^2(X,\mm),\,
		\norm{f_n-f}_{L^2}\to 0
\right\},
\end{equation}
where $\abs{\nabla f_n}(x)$ is the slope of $f_n$ at the point $x$.
Then $W^{1,2}(X,\sfd,\mm)$ is defined as the space of functions $f\in L^2(X,\mm)$ with finite Cheeger energy, endowed with the norm
\begin{equation*}
\|f\|_{W^{1,2}(X,\sfd,\mm)} : =  
\left\{\|f\|_{L^2(X,\mm)}+ \Ch(f)^{\frac{1}{2}}
\right\}
\end{equation*}
which makes $W^{1,2}(X,\sfd,\mm)$ a Banach space.
For any $f\in W^{1,2}(X,\sfd,\mm)$, one can single out a distinguished object 
$|\nabla f|_{w}\in L^2 (X, \mm)$, which plays the role of the modulus of the gradient and provides the integral representation
$$
\Ch(f)=\frac{1}{2} \int |\nabla f|_{w}^2 \, \mm;
$$
this function is called the \emph{minimal weak upper gradient} (after its identification with the minimal relaxed gradient).
For Lipschitz functions $\Lip(f)$ is a weak upper gradient for $f$.

\medskip 

Next we review the definition of Laplacian. 
Throughout the note, the various curvature conditions
we will assume will always imply $X$ to be proper (bounded and closed sets are compact) thus simplifying the presentation.

We recall that a \emph{Radon functional  over  
$X$} is a  linear functional $T:\LIP_{c}(X)\to \R$ such that  for every compact subset $W\subset X$ there exists a 
constant $C_{W}\geq 0$ so that
$$
|T(f)|\leq C_{W} \max_{W} |u|, \quad \text{for all } u\in \LIP_{c}(\Omega) \text{ with } \supp(u)\subset X.
$$ 
Finally for any $f,u$ locally in the Sobolev space (see \cite{Gigli12}), define  the functions $D^{\pm} f (\nabla u) : X \to \R$  by 
$$
D^{+} f (\nabla u) : = \inf_{\ve > 0} \frac{ |D (u + \ve f)|_{w}^{2} - |Du|_{w}^{2} }{2\ve},
$$
while $D^{-} f (\nabla u)$ is obtained replacing $\inf_{\ve>0}$ with $\sup_{\ve <0}$.

\begin{definition}[General definition]\label{D:measureLaplace}
Let $(X,\sfd,\mm)$ be a m.m.s. and $f : X \to \R$ be a Borel function. The function $f \in W^{1,2}(X,\sfd,\mm)$ is in the domain of the Laplacian in $X$, 
$f \in D(\bd)$, provided  there exists a Radon functional 
$T : \Lip_{c}(X) \to \R$ such that
$$
\int D^{-}u (\nabla f)\,\mm 
\leq - T( u) 
\leq
\int D^{+}u (\nabla f)\,\mm, 
$$
for each $u \in \Lip_{c}(X)$. In this case we write $T \in \bd(f)$.
\end{definition}

\begin{definition}[Eigenfunction]\label{D:eigenfunction}
Let $(X,\sfd,\mm)$ be a m.m.s. and $f$ be in $W^{1,2}(X,\sfd,\mm)$. The function $f$ is an eigenfunction for $-\bd$ if there 
exists $\lambda > 0$ such that 
$$
-\lambda f \mm \in \bd f;
$$ 
i.e. 
$$
\int D^{-}u (\nabla f) \,\mm
\leq \lambda \int f u \, \mm
\leq\int D^{+}u (\nabla f) \,\mm, 
$$
for each $u \in \Lip_{c}(X)$.
\end{definition}

\begin{remark}\label{R:zeromean}
It is straightforward to check that any eigenfunction 
has zero mean, provided $\mm(X)<\infty$. 
Here we only sketch the argument when $X$ 
is proper. Consider any 
sequence $(\chi_{n})$ of 1-Lipschitz functions with bounded support and values in $[0,1]$ such that $\chi_{n} \equiv 1$
in $B_{n}(\bar x)$, for some fixed $\bar x\in X$.  
Since  we are assuming $X$ to be proper, $\chi_{n}  \in \Lip_{c}(X)$ and therefore
$$
\int D^{-}\chi_{n} (\nabla f) \,\mm
\leq
\lambda \int \chi_{n} f \, \mm 
\leq
\int D^{+}\chi_{n} (\nabla f) \,\mm;
$$
for both quantities, 
$$
\left|\int D^{\pm} \chi_{n}(\nabla f) \, \mm \right| \leq 
\int_{X\setminus B_{n}(\bar x)} |\nabla f|_{w} \,\mm
$$
that are both converging to zero, provided $\mm(X) < \infty$, 
giving $\int f \, \mm = 0$ by dominated convergence theorem.
\end{remark}

From the lower semi-continuity and convexity of
$\Ch : D(\Ch) \subset L^{2}(X,\mm) \to [0,\infty)$, 
it is natural to consider an alternative definition of Laplacian 
related to the sub-differential $\partial^{-}\Ch$ of 
convex analysis. 
We recall that the sub-differential $\partial^{-}\Ch$ is the multivalued operator in $L^{2}(X,\mm)$ 
defined at all $f \in D(\Ch)$ by the family of inequalities
$$
g \in \partial^{-}\Ch(f) \iff \int_{X} g( h - f) \,\mm \leq \Ch(h) - 
\Ch(f), \quad \forall \ h\in L^{2}(X,\mm). 
$$

\begin{definition}[$L^{2}$-Laplacian]\label{D:L2Laplace}
The Laplacian  $-\Delta f \in L^{2}(X,\mm)$
of a function $f \in W^{1,2}(X,\sfd,\mm)$ is the element of 
minimal $L^{2}(X,\mm)$-norm in the sub-differential 
$\partial^{-}\Ch(f)$, provided the latter is non-empty.   
Accordingly, $f \in W^{1,2}(X,\sfd,\mm)$ is an eigenfunction 
provided $-\Delta f = \lambda f$, for some $\lambda > 0$.
\end{definition}

It will be clear from the proof of our results  
that the minimality requirement in the previous definition 
does not play any role: our main results 
will be valid for any element of the sub-differential. 
 
We also mention that at this level of generality, 
the two notions of Laplacian do not coincide,  
however (\cite[Proposition 4.9]{Gigli12}) if $f,g \in L^{2}(X,\mm)$ with $\Ch(f) < \infty$ and 
$-g \in \partial^{-}\Ch(f)$, then $g \in D(\bd)$ and 
$g\,\mm \in \bd f$.

\smallskip
\smallskip
 
Again from the lower semi-continuity and convexity of
$\Ch$, invoking the 
classical theory of gradient flows of convex functionals in Hilbert 
spaces, it follows that: for any $f \in L^{2}(X,\mm)$ there 
exists a unique continuous curve $(f_{t})_{t\geq 0}$ 
in $L^{2}(X,\mm)$ 
locally absolutely continuous in $[0,\infty)$ with $f_{0} = f$, 
such that $\frac{d}{dt} f_{t} \in \partial^{-}\Ch(f_{t})$ for 
a.e. $t > 0$.  
 
 The existence of the flow for all 
$f\in L^{2}(X,\mm)$ 
stems from the density of $D(\Ch)$ in $L^{2}(X, \mm)$. 
This gives rise to a semigroup $(H_{t})_{t\geq 0}$ on 
$L^{2}(X, \mm)$ defined by $H_{t} f = f_{t}$, 
where $f_{t}$ is the unique $L^{2}$-gradient flow of $\Ch$.

It follows that $f_{t} \in D(\Delta)$ and 
$$
\frac{d^{+}}{dt} f_{t} = \Delta f_{t}, \qquad \forall \ t\in (0,\infty),
$$
according to Definition \ref{D:L2Laplace}.

On the other hand, one can study the metric gradient flow of the Boltzmann entropy $\Ent$ in $\mathcal{P}_{2}(X, \sfd)$. 
If $(X, \sfd, \mm)$ 
satisfies $\CD(K,\infty)$, it has been proven in \cite{AGS11a} that 
for any $\mu \in D(\Ent)$ 
there exists a unique gradient flow of $\Ent$ 
starting from $\mu$ (for details we refer to \cite{AGS11a}).
This gives rise to a semigroup $(\mathfrak{H}_{t} )_{t \geq 0}$ 
on $P_{2}(X,\sfd)$ defined by $\mathfrak{H}_{t}\mu = \mu_{t}$
where $\mu_{t}$
is the unique gradient flow of $\Ent$ starting from $\mu$.

One of the main result of $\cite{AGS11a}$ is the identification of the two gradient flows: if $(X,\sfd,\mm)$ is a $\CD(K,\infty)$ space 
and $f \in L^{2}(X,\mm)$ such that $f\mm = \mu \in \mathcal{P}_{2}(X,\sfd)$, then 
\begin{equation}\label{E:identification}
\mathfrak{H}_{t} \mu = (H_{t}f) \, \mm, \qquad \forall \ t \geq 0.
\end{equation}
In particular we will only use the notation $H_{t}$ for both semi-groups.

\begin{definition}[$\RCD$ condition]
We say that $(X,\sfd,\mm)$ 
is \emph{infinitesimally Hilbertian} if the Cheeger energy 
$\Ch_2$ defined in \eqref{eq:cheeger_energy} is a quadratic form  on $W^{1,2}(X,\sfd,\mm)$. 
Finally $(X,\sfd,\mm)$ satisfies the 
$\RCD(K,N)$ condition if it satisfies the 
$\CD(K,N)$ condition and it is infinitesimally Hilbertian.
\end{definition}

Under the $\RCD$ condition, powerful 
contraction estimates for the heat flow are at disposal.

\begin{theorem}[Theorem 3 of \cite{EKS}]\label{T:contraction}
Let $(X,\sfd,\mm)$ be a metric measure space verifying 
$\RCD(K,N)$, then for any $\mu,\nu \in \mathcal{P}_{2}(X)$
and $s,t > 0$
\begin{equation}\label{E:contractionts}
W_{2}(H_{t}\mu, H_{s}\nu)^{2} 
\leq e^{-K \tau(s,t)} W_{2}(\mu,\nu)^{2} + 2N \frac{1- e^{-K \tau(s,t)}}{K \tau(s,t)} (\sqrt{t} -\sqrt{s})^{2},
\end{equation}
where $\tau(s,t) = 2(t + s + \sqrt{ts})/3$.
\end{theorem}

Also, if $X$ is infinitesimally Hilbertian then the two notions 
of Laplacian coincide and the previous implication can be reversed. 
Indeed 
if $f,g \in L^{2}(X,\mm)$ with $\Ch(f) < \infty$ and 
$f \in D(\bd)$ with $g \mm \in \bd f$ then
$- g \in \partial^{-}\Ch(f)$.

\smallskip
Finally if $W^{1,2}(X,\sfd,\mm)$ is an Hilbert space (hence in the $\RCD$ case) and $f \in W^{1,2}(X,\sfd,\mm)$ is an eigenfunction for the Laplacian 
in the sense of Definition \ref{D:eigenfunction}, 
then $H_{t}f  = e^{-\lambda t} f$.
It is indeed enough to check that 
$-\lambda e^{-\lambda t} f \in \partial^{-}\Ch( e^{-\lambda t} f )$
that is equivalent to 
$-\lambda f \in \partial^{-}\Ch( f )$ and this follows from 
\cite[Proposition 4.12]{Gigli12}.

\bigskip
\section{One dimensional indeterminacy estimates}\label{S:onedimensional}

In this section we will obtain the one-dimensional version of the uncertainty principle 
we will then integrate via Disintegration Theorem. 
A slightly different version of the following Proposition \ref{P:basic} 
was already present in the literature \cite[Theorem 4]{Steinerberger20}.

Throughout this section we will tacitly assume the Wasserstein distance to be defined on any couple of non-negative Borel measures having the same finite mass (not necessarily coinciding with 1). \\

We fix some notations that will be useful in this section and in the following one.
\\ 
Given a function $f:I\to\Real$ with zero mean, with $I$ real closed interval, possibly of infinite length, satisfying the hypotheses of Theorem \ref{T:localization} we say that the trasport of $f$ goes along a unique trasport ray if applying Theorem \ref{T:localization} one has that the partition in $\lbrace X_{\alpha}\rbrace_{\alpha}$ is made of only one element.\\
We  underline in addition that, recalling the notations of Lemma \ref{L:Perreal}, in the case of $f$ being a continuous function $B(\lbrace x\mid f(x)>0\rbrace)$ is a subset of the zero set of $f$. 	
\begin{proposition}\label{P:basic}
	Let $f \colon [0,1] \to \mathbb{R}$ be a continuous function having zero mean w.r.t Lebesgue measure, i.e. 
	$$
	\int_{(0,1)} f^{+}(x)\,dx = \int_{(0,1)} f^{-}(x)\,dx, 
	$$
	and assume that the transport of $f$ goes along a 
	unique transport ray, (see notations above). 
	Then it holds:
	\begin{equation}\label{E:est1}
	W_1(f^+\mathcal{L}^1,f^-\mathcal{L}^1) \,\mathcal{H}^0\left(B(\lbrace x\mid f(x)>0 \rbrace)\right)
	\geq
	\frac{\norm{f^+}^2_{L^1(0,1)}}
	{2\min \{ \|f^{+} \|_{L^{\infty}(0,1)}, \|f^{-} \|_{L^{\infty}(0,1)}\}}.
	\end{equation}
\end{proposition}
\begin{proof}
{\bf Step 1.}
We claim that given two non-negative functions  $f,g\in L^{\infty }(0,1)$ such that\\ $\int_{[0,1]}f(x)\,dx=\int_{[0,1]}g(x)\,dx$ and which satisfy the following condition on the supports: there exists $\bar x\in(0,1)$ such that 
\begin{equation}\label{hypoth}
\supp{\lbrace f\rbrace }\subseteq[0,\bar x],\quad
\supp{ \lbrace g \rbrace }\subseteq[\bar x,1],
\end{equation}
then one has
\begin{align}\label{ineq:base}
W_1(f,g)\geq \frac{1}{2}\frac{\|f\|^2_{L^{1}}}{\min\lbrace {\|f\|_{L^{\infty}},{\|g\|_{L^{\infty}}}\rbrace }}.
\end{align}
Indeed we can consider the two following rearrangement of the  masses 
$$
r_{f}\mathcal{L}^{1}:=\norm{f}_{L^{\infty}}\chi_{(\bar{x}-\tau_{f},\bar{x})}
\mathcal{L}^1,\quad 
r_{g}\mathcal{L}^{1}:=\norm{g}_{L^{\infty}}\chi_{(\bar{x},\bar{x} +\tau_{g})}
\mathcal{L}^1,
$$
with $\tau_{f}$ and $\tau_{g}$ chosen so that the total mass of 
$r_{f}\mathcal{L}^1$ is the same total mass of 
$f \mathcal{L}^1$, and the same for $r_{g}\mathcal{L}^1$ and 	$g\mathcal{L}^1$.
We notice that by direct calculation it holds 
\begin{equation}\label{eq:directcost}
W_1(r_f\mathcal{L}^{1},r_g\mathcal{L}^{1})
=\frac{1}{2}\left( \frac{\|f \|^{2}_{L^{1}(0,1)}}
{\|f \|_{L^{\infty}(0,1)}}+ 
\frac{\|g \|^{2}_{L^{1}(0,1)}}
{\|g \|_{L^{\infty}(0,1)}}\right),
\end{equation}
and then we observe  that 
\begin{equation}\label{ineq:bound}
W_1(f\mathcal{L}^1,g\mathcal{L}^1)\geq W_1(r_f\mathcal{L}^{1},r_g\mathcal{L}^{1}).
\end{equation}
Indeed	for any $\pi$ optimal transport plan between $f\mathcal{L}^1$ and $g\mathcal{L}^1$, 
one has
\begin{align*}
W_1(f\mathcal{L}^1,g\mathcal{L}^1) 
&	= \int |x-y|\pi(dxdy) =  \int_{(\bar x, 1)} (y - \bar x) g(y)\,dy 
+\int_{(0,\bar x)} (\bar x - x) f(x)\,dx\\ 
& = \int_{(\bar x, 1)}yg(y)\,dy +\int_{(0,\bar x)}-xf(x)\,dx\geq \int_{(\bar x, 1)}yr_g(y)\,dy+\int_{(0,\bar x)}-xr_f(x)\,dx
\end{align*}
where the last inequality follow from the two following observations:
\begin{itemize}
\item $g-r_g\leq 0 $ in $(\bar{x}, \bar{x}+\tau_{g})$ and $g-r_g\geq0 $ in $( \bar{x}+\tau_{g},1)$ ,
\item $f-r_f\leq 0$ in $(0,\bar{x}-\tau_{f})$ 	and $f-r_f\geq0$ in $(\bar{x}-\tau_{f},\bar{x})$,
\end{itemize}
and the fact that for a function $\psi:(0,+\infty) \to \Real$ with zero mean and such that $\psi\leq 0$ in $(0,a)$ and $\psi\geq 0$ in $(a,+\infty)$ it holds that 	$\int_{(0,+\infty)}x\psi(x)\,dx\geq 0$. So finally putting \eqref{eq:directcost} and \eqref{ineq:bound} togheter  we obtain 
$$
W_1(f\mathcal{L}^1,g\mathcal{L}^1) \geq  \frac{\|f\|^{2}_{L^{1}(0,1)}}{2 \min \{ \|f \|_{L^{\infty}(0,1)}, \|g \|_{L^{\infty}(0,1)}\}} .
$$
{\bf Step 2.}
Let $f:[0,1]\to \Real$ be continuous and such that $\int_{(0,1)}f(x)\,dx=0$.\\
Let $C_1,\dots,C_n$ be the connected components of $\lbrace x\in  [0,1]\mid f(x)>0 \rbrace$, with $n$ possibly $+\infty$. As in Lemma \ref{L:Perreal} we consider the set $\cup_{k=0}^{n}\bar{C}_k$. We observe that it is the union of disjoint closed intervals $D_k$: $\cup_{k=0}^{n}\bar{C}_k=\cup_{k=0}^{m}D_k$. If $m=+\infty$  then  $\mathcal{H}^{0}(B(\lbrace x\mid f(x)>0 \rbrace))=+\infty$ and the statement is trivially true. Hence we can assume that $m<+\infty$.  Let $T:[0,1]\to \Real$  be an optimal transport map for the problem. We  observe that $T_{\#}(f^{+}\llcorner_{D_k}\mathcal{L}^1)\leq f^-\mathcal{L}^1$, so in particular it is absolutely continuous with respect to the Lebesgue measure and its density $\frac{dT_{\sharp}(f^{+}\llcorner_{D_k}\mathcal{L}^1)}{d\mathcal{L}^1}$ satisfies  $\|\frac{dT_{\sharp}(f^{+}\llcorner_{D_k}\mathcal{L}^1)}{d\mathcal{L}^1}\|_{L^{\infty}}\leq \|f^{-}\|_{L^{\infty}}$. In addition we observe that since the transport of $f$ goes along a unique transport ray,  we have that either  $u(x):=-x$ or $u(x):=x$ is a Kantorovich potential for the problem, so assuming without loss of generality $u(x)=-x$. Using the definition of Kantorovich potential we have that each couple $(x,T(x))$ with $x\in\supp{f^+}$ satisfies $u(x)-u(T(x))=|x-T(x)|$ so in particular $T(x)=x+|x-T(x)|$ and $T(x)\geq x$. This means that for each $k$ the couple $f^{+}\llcorner_{D_k}$ and $\frac{dT_{\sharp}(f^{+}\llcorner_{D_k}\mathcal{L}^1)}{d\mathcal{L}^1}$ satisfies the hypotheses of the previous step. This is because the second function is concentrated on $T(D_k)$.  
So in particular applying the previous estimate to each of $f^{+}\llcorner_{D_k}\mathcal{L}^1$ and its pushforward through the map  $T$ one has that
\begin{align*}
W_{1}(f^{+}\llcorner_{D_k}\mathcal{L}^{1},T_{\sharp}(f^{+}\llcorner_{D_k}\mathcal{L}^1))\geq \frac{1}{2}\frac{\|f^{+}\llcorner_{D_k}\|^2_{L^1(0,1)}}{\min\lbrace{\|f^{+}\|_{L^{\infty}(0,1)}, \|f^{-}\|_{L^{\infty}(0,1)}\rbrace }}.
\end{align*}
Being in addition the sets $C_k$ disjoint, we have that 
$$
W_{1}(f^{+}\mathcal{L}^1,f^{-}\mathcal{L}^1) = \sum_{k = 1}^{m}W_{1}(f^{+}\llcorner_{D_k}\mathcal{L}^{1},T_{\sharp}(f^{+}\llcorner_{D_k}\mathcal{L}^1)))\geq\frac{1}{2}\sum_{k = 1}^{m}\frac{\|f^{+}\llcorner_{D_k}\|^2_{L^1(0,1)}}{\min{\lbrace \|f^{+}\|_{L^{\infty}(0,1)},\|f^{-}\|_{L^{\infty}(0,1)}\rbrace}}. 
$$
Applying Cauchy-Schwartz inequality we get 
\begin{align*}
W_{1}(f^{+}\mathcal{L}^1,f^{-}\mathcal{L}^1)&\geq \frac{1}{2 \min{\lbrace \|f^{+}\|_{L^{\infty}(0,1)},\|f^{-}\|_{L^{\infty}(0,1)}\rbrace}}\frac{1}{m}\left(\sum_{k=1}^{m}\|f^{+}\llcorner_{D_k}\|_{L^1(0,1)}\right)^2\\&=\frac{1}{2m}\frac{\|f^+\|^2_{L^1(0,1)}}{\min{\lbrace \|f^{+}\|_{L^{\infty}(0,1)},\|f^{-}\|_{L^{\infty}(0,1)}\rbrace}}.
\end{align*}
\end{proof}

\begin{remark}
	We observe that in the preceeding proposition the fact that the interval in which we are working in is exactly $[0,1]$ plays no role, so it analogously holds for an interval $[a,b]$ o in general for intervals of infinite length provided that the function $f$ is in $L^1$.
\end{remark}


\subsection{One dimensional densities with curvature bounds}

We now obtain the one-dimensional estimate also for a reference measure other than the Lebesgue one.

As before, we will first consider the case of functions 
defined on a compact interval $[0,D]$ and then we will discuss the non-compact case in the following Remark \ref{R:noncompact}.

\begin{proposition}\label{P:density}
Let $h\colon [0,D] \to [0,+\infty)$ be a $\CD(K,\infty)$-density (recall Definition \ref{def:CDKN-density}). 
\\
Let $f \colon [0,D] \to \mathbb{R}$ be a continuous function having zero mean w.r.t 
the measure $h\mathcal{L}^{1}$: $\int_{(0,D)} f(x) h(x)\,dx = \nobreak0$.
Assume also that the transport of $f h$ goes along a unique transport ray. 
Then it holds
\begin{equation}\label{E:est3}
W_1(f^+h\mathcal{L}^1,f^-h\mathcal{L}^1) \left(\sum_{ x\in B(\lbrace  f>0 \rbrace) }  h(x)\right)\geq \frac{\norm{fh}^2_{L^1(0,D)}}{8C_{K,D}{\norm{f}_{L^{\infty}(0,D)}}},
\end{equation}
(see Lemma \ref{L:Perreal} for the definition of $B(\lbrace f>0\rbrace)$) where
\begin{equation}\label{E:coefficient}
C_{K,D}: = 
\begin{cases} 
1 & K \geq 0, \\ e^{-K D^{2}/2} & K < 0.
\end{cases}
\end{equation}
\end{proposition}

\begin{proof}			
{\bf Step 1.}
We make the following preliminary observation. 
From $\CD(K,\infty)$ assumption it follows that the map 
$$
[0,D]\ni x \mapsto \log h(x) + K \frac{(x-\bar x)^{2}}{2},
$$
is concave. In particular, for each $\bar x\in(0,D)$ either is increasing in $[0,\bar{x}]$ 
or is decreasing in $[\bar{x},D]$.  Hence in the first case 
$$
\log h(x) + K \frac{(x-\bar x)^{2}}{2} \leq \log h(\bar x), \qquad \forall \ x \in [0,\bar x]; 
$$
while in the second case: 
$$
\log h(x) + K \frac{(x-\bar x)^{2}}{2} \leq \log h(\bar x), \qquad \forall \ x \in [\bar x,D]; 
$$
The combination of the two previous inequalities yields
\begin{equation}\label{E:intermediateh}
\min \{ \| h \|_{L^{\infty}[0,\bar x]}, \| h \|_{L^{\infty}[\bar x,D]}  \}
\leq h(\bar x) C_{K,D}.
\end{equation}
where $C_{K,D}$ is the defined in \eqref{E:coefficient}.
\\
Similarly to { Step 1} of the previous proof we make a base estimate that we will use in the next step: we take two non negative bounded functions $f, g:[0,D]\to \Real$ such that $\int_{[0,D]}fh\,dx=\int_{[0,D]}gh\,dx$, satisfying %
\begin{equation}
\supp{\lbrace f\rbrace }\subseteq[0,\bar x],\quad\supp{ \lbrace g \rbrace }\subseteq[\bar x,D].
\end{equation}
We can now apply \eqref{ineq:base} to $fh, gh$ (recalling that $h$ is positive) and
\begin{align}\label{E:claimsim}
W_1(fh\mathcal{L}^1,gh\mathcal{L}^1)\geq &~ \frac{\norm{f h}^2_{L^1(0,D)}}{2\min\{{{\norm{fh}_{L^{\infty}(0,D)}},{\norm{gh}_{L^{\infty}(0,D)}}}\}}\nonumber \\
\geq &~ \frac{\norm{fh}^2_{L^1(0,D)}}{2C_{K,D} \max\{{{\norm{f}_{L^{\infty}(0,D)}},{\norm{g}_{L^{\infty}(0,D)}}}\}h(\bar{x})},
\end{align}
	where the second inequality follows from \eqref{E:intermediateh}.
	\smallskip	\\				
	{\bf Step 2.} Consider $C_1,\dots,C_n$ the connected components of $\lbrace f>0\rbrace $ with $n$ possibly $+\infty$. As in Lemma \ref{L:Perreal} we consider the set $\cup_{k=0}^{n}\bar{C}_k$. We observe that it is the union of disjoint closed intervals: $\cup_{k=0}^{n}\bar{C}_k=\cup_{k=0}^{m}[a_k,b_k]$,  with $m$ possibly $+\infty$.  
	We will proceed as in the proof of Proposition \ref{P:basic}:
	we consider an optimal trasport map $T$ and we obtain
	$$
	W_{1}(f^{+}h\mathcal{L}^1,f^{-}h\mathcal{L}^1) = \sum_{k = 0}^{m}W_{1}(f^{+}h\mathcal{L}^1\llcorner_{(a_{k}, b_{k})},
	T_{\sharp}(f^{+}h\mathcal{L}^1\llcorner_{(a_{k}, b_{k})})).
	$$
	Then we can apply \eqref{E:claimsim} (as in the previous proof using the fact that the transport of $f^+h$ into $f^-h$ goes along a unique transport ray) to obtain:
	\begin{align*}
	W_1(f^+h\mathcal{L}^1,f^-h\mathcal{L}^1) = &~ 
	\sum_{k = 0}^{m}W_{1}(f^{+}h\mathcal{L}^1\llcorner_{(a_{k},b_{k})},
	T_{\sharp}(f^{+}h\mathcal{L}^1\llcorner_{(a_{k}, b_{k})})) \\ 
	\geq &~ \sum_{k = 0}^{m}\frac{\norm{f^+h}^2_{L^1(a_{k}, b_{k})}}
	{2C_{K,D}\norm{f}_{L^{\infty}(a_{k}, b_{k})}(h(a_{k})+h(b_{k}))} \\ 
	\geq &~ \frac{1}{2C_{K,D}\norm{f}_{L^{\infty}(0,D)}}
	\sum_{k = 0}^{m}\frac{\norm{f^+h}^2_{L^1(a_{k}, b_{k})}}{(h(a_{k})+h(b_{k}))} \\
	\geq &~ \frac{\norm{fh}^2_{L^1(0,D)}}{8C_{K,D}\norm{f}_{L^{\infty}(0,D)}\sum_{k = 0}^{m}(h(a_{k})+h(b_{k}))},
	\end{align*} 
	with the convention that if $a_k=0$ (resp. $b_k=D$) the term $h(a_k)$ (resp. $h(b_k)$) does not appear.
	From this we get 
	$$
	8C_{K,D} W_1(f^+h\mathcal{L}^1,f^-h\mathcal{L}^1)\left(\sum_{k=0}^{m}(h(a_{k})+h(b_{k}))\right)
	\geq \frac{\norm{fh}_{L^1(0,D)}^2}{\norm{f}_{L^{\infty}(0,D)}},
	$$
	 with the same convention on $h(0)$, $h(D)$ as above, from which the conclusion follows.
\end{proof}

\begin{remark}\label{R:noncompact}
	The case of non-compact intervals of definition holds without modifications. The only relevant case is $K \geq 0$ and $D = \infty$ indeed for $K < 0$ and $D = \infty$, the claim becomes empty.	Notice that $D$ plays a role only in \eqref{E:intermediateh} where, in the relevant cases, it becomes independent on $D$.
\end{remark}

\subsection{The case of $\MCP(K,N)$ densities}

We now address the case of an $\MCP(K,N)$-density.
As it is clear from the proof of Proposition \ref{P:density}, the only place where 
the $\CD(K,\infty)$ assumption has been used is to 
ensure $h > 0$ over $(0,D)$ and to derive \eqref{E:intermediateh}. 
A similar estimate, with suitable
variations, can be obtained also for $\MCP(K,N)$-densities.

\begin{lemma}\label{L:densityMCP}
Let $h : [0,D] \to [0,\infty]$ be an $\MCP(K,N)$-density for some real parameters $K,N$
with $N \geq 1$. 
Then for any $\bar x \in [0,D]$ the following estimates holds true:
\begin{equation}\label{E:intermediatehMCP}
\min \{ \| h \|_{L^{\infty}[0,\bar x]}, \| h \|_{L^{\infty}[\bar x,D]}  \}
\leq h(\bar x) C_{K,N,D},
\end{equation}
where 
\begin{equation}\label{E:constantMCP}
C_{K,N,D} : = 
\begin{cases} 2^{N-1} &  K \geq 0 \\ 
2^{N-1}e^{\sqrt{-K(N-1)}\frac{D}{2}} & K < 0. 
\end{cases}
\end{equation}
\end{lemma} 

\begin{proof}
The claim will follow from simple manipulations of 
\eqref{E:MCPdef2}. For clarity we recall it: for all $0 \leq x_{0} \leq x_{1} \leq D$
$$
\left( \frac{s_{K/(N-1)}(D - x_{1}  )}{s_{K/(N-1)}(D - x_{0}  )} \right)^{N-1} 
\leq \frac{h(x_{1} ) }{h (x_{0})} \leq 
\left( \frac{s_{K/(N-1)}( x_{1} )}{s_{K/(N-1)}( x_{0} )} \right)^{N-1}; 
$$
indeed for $x \in [0,\bar x]$ 
$$
h(x) \leq \left( \frac{s_{K/(N-1)}(D - x)}{s_{K/(N-1)}(D - \bar x)} \right)^{N-1} 
h(\bar x) \leq \frac{h(\bar x)}{s_{K/(N-1)}(D - \bar x)^{N-1}} \sup_{0\leq x \leq \bar x} s_{K/(N-1)}(D - x)^{N-1}
$$
and for $x \in [\bar x,D]$
$$
h(x) \leq \left( \frac{s_{K/(N-1)}( x  )}{s_{K/(N-1)}( \bar x  )} \right)^{N-1} 
h(\bar x)
\leq \frac{h(\bar x)}{s_{K/(N-1)}(\bar x)^{N-1}} \sup_{\bar x\leq x \leq D} s_{K/(N-1)}(x)^{N-1}.
$$
Then if $K  \geq 0$, in particular $h$ will be $\MCP(0,N)$ giving  
$$
\sup_{0\leq x \leq \bar x} h(x) \leq h(\bar x) \left(\frac{D}{D-\bar x}\right)^{N-1}, 
\qquad
\sup_{\bar x \leq x \leq D} h(x) \leq h(\bar x) \left(\frac{D}{\bar x}\right)^{N-1}, 
$$
and therefore 
$$
\min \{ \| h \|_{L^{\infty}[0,\bar x]}, \| h \|_{L^{\infty}[\bar x,D]}  \}
\leq h(\bar x) D^{N-1} \min\{ 1/(D-\bar x), 1/ \bar x \}^{N-1} \leq  2^{N-1 }h(\bar x), 
$$
proving the inequality if $K \geq 0$. If $K < 0$, arguing analogously one gets
$$
\min \{ \| h \|_{L^{\infty}[0,\bar x]}, \| h \|_{L^{\infty}[\bar x,D]}  \}
\leq h(\bar x) 2^{N-1} e^{\sqrt{-K(N-1)}\frac{D}{2}},
$$
concluding the proof.
\end{proof}

Putting together the proof of Proposition \ref{P:density} and Lemma \ref{L:densityMCP} 
we straightforwardly obtain the next

\begin{proposition}\label{P:densityMCP}
Let $h\colon [0,D] \to [0,+\infty)$ be an $\MCP(K,N)$-density. 
Let $f \colon [0,D] \to \mathbb{R}$ be a continuous function having zero mean w.r.t 
the measure with density $h$: $\int_{(0,D)} f(x) h(x)\,dx = 0$.
Assume also that the transport of $f h$ goes along a unique transport ray:
$\int_{(0,s)} f(x) h (x)\,dx \geq 0$ for all $s \in [0,D]$.
Then it holds
\begin{equation}\label{E:est3}
W_1(f^+h\mathcal{L}^1,f^-h\mathcal{L}^1) \left(\sum_{\{ x\in B(\lbrace  f>0 \rbrace) \}}h(x)\right)
\geq \frac{\norm{fh}^2_{L^1(0,D)}}{8 C_{K,N,D}{\norm{f}_{L^{\infty}(0,D)}}}, 
\end{equation}
where $C_{K,N,D}$ is given by \eqref{E:constantMCP}.
\end{proposition}

\begin{remark}\label{R:noncompactMCP}
The case of non-compact intervals of definition holds 
again without modifications. 
The only relevant case here will be $K = 0$ and 
$D = \infty$; if $K > 0$, then $\MCP$ implies that 
$D < D_{K,\NN}$ (see \eqref{E:diameter}) 
while if $K < 0$ and  $D = \infty$, the claim becomes empty.
Notice that $D$ plays a role only in 
\eqref{E:intermediateh} that is the content of 
Lemma \ref{L:densityMCP}.
\end{remark}

\smallskip

\section{Indeterminacy estimates for metric measure spaces}\label{S:multid}
We now use the one-dimensional estimates of the previous section to deduce the following sharp indeterminacy estimates.

\begin{theorem}\label{T:multidim} 
Let $K,K',N \in \R$ with $N > 1$.
Let $(X,\sfd,\mm)$ be an essentially non-branching m.m.s. 
satisfying either $\CD(K,N)$ 
or $\MCP(K',N)$ and $\CD(K,\infty)$.
Let $f \in L^{1}(X,\mm)$ a continuous function or, 
alternatively, $f \in W^{1,2}(X,\sfd,\mm)$ be such that 
$\int_{X} f\, \mm = 0$. 
Assume also the existence of $x_{0} \in X$ such that $\int_{X} | f(x) |\,  \sfd(x,x_{0})\, \mm(dx)< \infty$. 
Then the following indeterminacy estimate is valid:
\begin{equation}\label{E:indeterminacyCD}
W_{1}(f^+\m,f^-\m)\cdot \Per\left(\{ x\in X\colon f(x)>0\}\right)
\geq \frac{\norm{f}^2_{L^1(X,\mm)}}{8C_{K,D}\|f \|_{L^{\infty}(X,\mm)}},
\end{equation}
where $D = \diam(X)$ and
$$
C_{K,D}: = 
\begin{cases} 
1 & K \geq 0, \\ e^{-K D^{2}/2} & K < 0.
\end{cases}
$$
\end{theorem}

\begin{remark}
Notice that curvature assumptions $\CD(K,N)$ and $\MCP(K,N)$ imply $D < \infty$ only in the range $K > 0$ and $N \in (1,\infty)$. 
Hence under the second set of assumptions ($\MCP(K',N)$ and $\CD(K,\infty)$), the result \eqref{E:indeterminacyCD} for $K \geq 0$ gives a non-trivial bound 
also in the non-compact case $D = \infty$.
\end{remark}

\begin{proof}
Given $f$ as in the assumptions, 
we can invoke localization paradigm (Theorem \ref{T:localization} 
and Theorem \ref{T:localizationinfty})
yielding a decomposition of the space $X$ as $X=Z\cup \mathcal{T}$, where $f$ is zero $\mm$-a.e. in $Z$ and $\mathcal{T}$ can be partitioned into $\{X_\alpha\}_\alpha$ with $\alpha$ 
in a Borel set $Q\subset X$, 
and a disintegration of $\mm$, 
$$
\mm\llcorner_{\T} = \int_{Q} \mm_{\alpha}\,\qq(d\alpha),   
$$
with $\qq$ Borel probability measure with $\qq(Q) = 1$ and 
$Q \ni \alpha \mapsto \mm_{\alpha} \in \mathcal{M}_{+}(X)$ 
satisfying the properties of Theorem \ref{T:localization}; in particular, 
$(X_\alpha,\sfd,\mm_\alpha)$ is a $\CD(K,N)$ space (or $\CD(K,\infty)$ see Theorem \ref{T:localizationinfty}), 
$\int_{X_\alpha}f\, \mm_{\alpha}=0$ and every $X_\alpha$ is a transport ray associated to the $L^1$-optimal trasport of $f^+\m$ into $f^-\m$.

\smallskip
{\bf Step 1.} 
As proven in \cite[Proposition 4.4]{CM18} for the case of signed distance functions, $\qq$ can be identified with 
a test plan, see \cite[Definition 5.1]{AGS11a}; hence,
if $f \in W^{1,2}(X,\sfd,\mm)$, 
by the identification between different definitions of Sobolev spaces \cite[Theorem 6.2]{AGS11a}, for $\qq$-a.e. $\alpha$ the function 
$f$ restricted to the geodesic $X_{\alpha}$  is Sobolev and therefore continuous.

As said in Remark \ref{R:realinterval}, we have an isomorphism between each space 
$(X_\alpha,\sfd,\mm_\alpha)$ and  spaces 
$(I_\alpha,\,|\cdot|,\, h_{\alpha} \cdot \mathcal{L}^{1})$, 
with $I_{\alpha}$ a real interval (of possible infinite length) satisfying the same $\CD(K,N)$ (or $\CD(K,\infty)$) condition, 
$\int_{I_\alpha} f_{\alpha}(x) h_\alpha(x)\,dx=0$  being $f_\alpha$ the corresponding of $f\llcorner_{X_\alpha}$ through the isomorphism and $I_\alpha$ transport ray for $f_\alpha$. 
Whenever possible, for simplicity 
of notation, we will use $f = f_{\alpha}$.

So now we can apply Proposition \ref{P:density} and we have that 
$\qq$-a.e. $\alpha \in Q$ it holds
\begin{equation}\label{E:1d}
W_1(f_{\alpha}^+h_\alpha\mathcal{L}^{1},f_{\alpha}^-h_\alpha\mathcal{L}^{1})\left(\sum_{ x\in B(\lbrace f_{\alpha}>0 \rbrace)} h_\alpha(x)\right)\geq 
\frac{\|f\|^2_{L^{1}(X_{\alpha}, \mm_{\alpha})}}
{8C_{K,D}\|f \|_{L^{\infty}(X_\alpha,\mm_{\alpha})}}. 
\end{equation}
By Lemma \ref{L:Perreal} 
$\sum_{ x\in B(\lbrace f_{\alpha}>0 \rbrace)}h_\alpha(x)
=\Per_{h_{\alpha}}(\{x\in I_\alpha \colon\, f_{\alpha}(x)>0\})$, 
hence using the isomorphisms of metric measure spaces, we have
\begin{equation*}
W_1(f^+\mm_\alpha,f^-\mm_\alpha) \,
 \Per_{\alpha}\left(\{x\in X_\alpha \colon\, f(x)>0\}\right)\geq \frac{\norm{f}^2_{L^{1}(X_{\alpha},\mm_{\alpha})}}
 {8C_{K,D}\|f \|_{L^{\infty}(X_\alpha,\mm_\alpha)}},
\end{equation*}
where $\Per_{\alpha}$ is the perimeter in $(X_{\alpha},\sfd, \mm_{\alpha})$ and $\Per_{h_{\alpha}}$  in $(I_\alpha,\,|\cdot|,\, h_{\alpha} \cdot \mathcal{L}^{1})$.
In the previous factor we have tacitly used that 
$C_{K,D}\geq C_{K,D_{\alpha}}$, where $D_{\alpha}$ is the length of 
$X_{\alpha}$.
 Integrating the square root of the inequality  with respect to the measure ${\sf q}$ on $Q$ and applying Holder inequality, we get 
\begin{align*}
 &\left(\int_{Q} W_1(f^+\mm_{\alpha},f^{-}\mm_{\alpha}) 
 \,\qq(d\alpha)\right)^{\frac{1}{2}}
\left(\int_{Q}\Per_{\alpha}(\{x\in X_{\alpha} \colon\, f(x)>0\})\,\qq(d\alpha)\right)^{\frac{1}{2}}\\
&\geq   
\int_{Q}\left(W_1(f^+\mm_{\alpha},f^{-}\mm_{\alpha})
\cdot \Per_{\alpha}(\{x\in X_{\alpha} \colon\, f(x)>0\})\right)^{\frac{1}{2}}\,\qq(d\alpha)\\ 
&\geq\bigintsss_{Q}\frac{\norm{f}_{L^1(X_{\alpha},\mm_\alpha)}}{2\sqrt{2C_{K,D}}\|f \|_{L^{\infty}(X_\alpha,\mm_{\alpha})}^{\frac{1}{2}}}\, \qq(d\alpha) \\
& \geq
\frac{1}{2\sqrt{2C_{K,D}}\|f \|^{\frac{1}{2}}_{L^{\infty}(X,\mm)}}\int_{Q}\int_{X_\alpha}|f(x)|\mm_\alpha(dx)\,\qq(d\alpha) \\
&=
\frac{\|f \|_{L^{1}(X,\mm)}}{2\sqrt{2C_{K,D}}\|f \|^{\frac{1}{2}}_{L^{\infty}(X,\mm)}}.
\end{align*}
Clearly 
$\int_{Q} W_{1}(f^+\mm_\alpha,f^{-}\mm_\alpha)\, \qq(d\alpha)
=W_{1}(f^+ \mm,f^- \mm)$;
therefore
\begin{equation*}
W_{1}(f^+\mm,f^-\mm )^{\frac{1}{2}}
\left(\int_{Q}\Per_{\alpha}(\{x\in X_\alpha \colon \,f(x)>0\})\,\qq(d\alpha)\right)^{\frac{1}{2}}\geq \frac{\|f \|_{L^{1}(X,\mm)}}{2\sqrt{2C_{K,D}}\|f \|^{\frac{1}{2}}_{L^{\infty}(X,\mm)}}.
\end{equation*}
The conclusion follows using Lemma \ref{L:perineq}.
\end{proof}

Repeating the same argument of the previous proof 
and using Proposition \ref{P:densityMCP}, we also obtain the analogous estimate
for spaces verifying the weaker $\MCP(K,N)$; as expected, weaker curvature assumptions yields a dependence on the dimension of the estimate.

\begin{theorem}\label{T:multidimMCP}
Let $K,N \in \R$ with $N > 1$.
Let $(X,\sfd,\mm)$ be an essentially non-branching m.m.s. 
verifying $\MCP(K,N)$. 

Let $f \in L^{1}(X,\mm)$ a continuous function or, 
alternatively, $f \in W^{1,2}(X,\sfd,\mm)$ be such that 
$\int_{X} f\, \mm = 0$. 
Assume also the existence of $x_{0} \in X$ such that $\int_{X} | f(x) |\,  \sfd(x,x_{0})\, \mm(dx)< \infty$. 
Then the following indeterminacy estimate is valid:
\begin{equation}\label{E:indeterminacyMCP}
W_{1}(f^+\m,f^-\m)\cdot \Per\left(\{ x\in X\colon f(x)>0\}\right)
\geq \frac{\norm{f}^2_{L^1(X,\mm)}}{8C_{K,N,D}\|f \|_{L^{\infty}(X,\mm)}},
\end{equation}
where $\diam(X)=D$ and 
$$
C_{K,N,D} : = 
\begin{cases} 2^{N-1} &  K \geq 0, \\ 
2^{N-1} e^{\sqrt{-K(N-1)}\frac{D}{2}} & K < 0. 
\end{cases}
$$
\end{theorem}

\medskip

\section{Nodal Sets of Eigenfunctions}\label{S:Nodal}
The plan for this section is to 
obtain lower bounds on the nodal set of eigenfunctions 
under curvature assumptions.  
Building on the previous Theorem \ref{T:multidim} 
and Theorem \ref{T:multidimMCP}, this will reduce to 
find an upper bound on the $W_{1}$ distance between the positive and the negative part of the eigenfunction.

\subsection{Nodal set under $\MCP$ and $\CD$}

Here, as throughout the paper, the $W_{1}$ distance 
is understood to be tacitly extended between any finite non-negative measure with the same total mass.

\begin{lemma}\label{L:nodalestimate}
Let $(X,\sfd,\mm)$ be a m.m.s. verifying $\MCP(K,N)$ and 
with finite total mass, $\mm(X) < \infty$.
Let $f$ be an eigenfunction of the Laplacian with eigenvalue $\lambda\neq 0$ accordingly to Definition \ref{D:L2Laplace}
and assume moreover  the existence of $x_{0} \in X$ such that $\int_{X} | f(x) |\,  \sfd(x,x_{0})\, \mm(dx)< \infty$.  

Then
$$
W_{1}(f^+\mm,f^-\mm)\leq \frac{\sqrt{\mm(X)}}{\sqrt{\lambda}}\|f\|_{L^2(X,\mm)}
$$
\end{lemma}

\begin{proof}
First from Remark \ref{R:zeromean}, $\int f \, \mm = 0$ 
and, by definition, $f \in W^{1,2}(X,\sfd,\mm)$.
By assumption Kantorovich duality has a solution and 
therefore exists a $1$-Lipschitz 
Kantorovich Potential $u: X\to \Real$ such that 
\begin{equation}\label{eq:transp}
W_{1}(f^+\mm,f^-\mm)=\int_{X}(f^+(x)-f^-(x))u(x)\,\mm(dx)=\int_{X}f(x)u(x) \,\mm(dx).
\end{equation}
Since $f$ is a eigenfunction in the sense of 
Definition \ref{D:L2Laplace}, then the following integration by-parts formula 
\begin{equation*}
\int_{X}D^-g(\nabla f)\,\mm \leq  \lambda \int_{X}g f\, \mm 
\leq \int_{X}D^{+}g(\nabla f)\, \mm, 
\end{equation*}
is valid for any $g \in W^{1,2}(X,\sfd,\mm)$ (see for instance 
the proof of \cite[Proposition 4.9]{Gigli12}). 

From $\mm(X)< \infty$ it follows that
$u \in W^{1,2}(X,\sfd,\mm)$, hence 
together with \eqref{eq:transp} gives
$$
W_{1}(f^+\mm,f^-\mm)\leq \frac{1}{\lambda}\int_{X}D^+u(\nabla f) \, \mm \leq \frac{1}{\lambda}\int_{X}\abs{Du}_{w}\abs{Df}_{w}\, \mm 
\leq  \frac{\Lip(u)}{\lambda}\int_{X}\abs{Df}_{w}\, \mm ,
$$
where we used the fact that $\abs{D^{\pm}u(\nabla f)}\leq \abs{Du}_{w}\abs{Df}_{w}$ and that $\Lip(u)$ is a weak upper gradient for $u$. 
Then by Holder inequality 
we have 
$$
\int_{X}\abs{Df}_{w}\,\mm\leq \mm(X)^{\frac{1}{2}}\left(\int_{X}\abs{Df}^2_w\, \mm\right)^{\frac{1}{2}}
=
\mm(X)^{\frac{1}{2}} \left(\int_{X}D^-f(\nabla f)\, \mm\right)^{\frac{1}{2}}
\leq  
\mm(X)^{\frac{1}{2}}\sqrt{\lambda}\left(\int_{X}f^2\,\mm \right)^{\frac{1}{2}}
$$
noticing that $ Df^+(\nabla f)=\abs{Df}^2_w$ 
(see \cite[(3.6)]{Gigli12}) and $f$ itself as test-function.
\end{proof}

\begin{remark}
The same claim can be obtained assuming 
$f$ to be an eigenfunction for the more general notion of Laplacian of Definition \ref{D:measureLaplace}, provided one additionally 
knows $f$ to be Lipschitz regular, yielding integration by-parts 
formula against any Sobolev functions (and in particular yielding 
$f$ to be an eigenfunction for the Laplacian of Definition \ref{D:L2Laplace}.) 
\end{remark}

Putting together Lemma \ref{L:nodalestimate} and 
the previous results we obtain the next

\begin{theorem}\label{T:nodalCD}
Let $(X,\sfd,\mm)$ be an essentially non-branching  m.m.s. 
 verifying either $\CD(K,N)$ 
or $\MCP(K',N')$ and $\CD(K,\infty)$ and such that 
$\mm(X) < \infty$.

Let $f$ be an eigenfunction of the Laplacian of eigenvalue 
$\lambda> 0$ accordingly to 
to Definition \ref{D:L2Laplace} 
and assume moreover  the existence of $x_{0} \in X$ such that $\int_{X} | f(x) |\,  \sfd(x,x_{0})\, \mm(dx)< \infty$.  
Then the following estimate on the size of the its nodal set holds true:
\begin{equation*}
\Per(\{x\in X : f(x)>0 \})\geq   
\frac{\sqrt{\lambda}}{8C_{K,D}\sqrt{\mm(X)}}  
\cdot \frac{\|f\|^2_{L^1(X,\mm)}}{\|f\|_{L^{2}(X,\mm)}\|f\|_{L^{\infty}(X,\mm)}},
\end{equation*}
where $D = \diam(X)$ and
$$
C_{K,D}: = 
\begin{cases} 
1 & K \geq 0, \\ e^{-K D^{2}/2} & K < 0.
\end{cases}
$$
\end{theorem}

\begin{proof}
Theorem \ref{T:multidim} and Lemma \ref{L:nodalestimate} imply the claim. 
\end{proof}

Using Theorem \ref{T:multidimMCP}, we obtain the following analogous statement 
for spaces verifying the weaker $\MCP(K,N)$ condition with dimension-dependent 
constant appearing. The proof, being completely the same is omitted.

\begin{theorem}\label{T:nodalMCP}
Let $(X,\sfd,\mm)$ be an essentially non-branching  m.m.s. 
 verifying $\MCP(K,N)$ and such that 
$\mm(X) < \infty$. 

Let $f$ be an eigenfunction of the Laplacian of eigenvalue 
$\lambda> 0$ accordingly to 
to Definition \ref{D:L2Laplace} 
and assume moreover  the existence of $x_{0} \in X$ such that $\int_{X} | f(x) |\,  \sfd(x,x_{0})\, \mm(dx)< \infty$.  

Then the following estimate on the size of the its nodal set holds true:
\begin{equation*}
\Per(\{x\in X : f(x)>0 \})\geq  
\frac{ \sqrt{\lambda}}{8C_{K,N,D}\sqrt{\mm(X)}}  
\cdot \frac{\|f\|^2_{L^1(X,\mm)}}{\|f\|_{L^{2}(X,\mm)}\|f\|_{L^{\infty}(X,\mm)}},
\end{equation*}
where $D = \diam(X)$ and 
$$
C_{K,N,D} : = 
\begin{cases} 2^{N-1} &  K \geq 0, \\ 
2^{N-1} e^{\sqrt{-K(N-1)}\frac{D}{2}} & K < 0. 
\end{cases}
$$
\end{theorem}

\subsection{The infinitesimally Hilbertian case}
Assuming the heat flow to be linear yields 
more sophisticated argument and sharper estimtes.
We start with the following folklore result 
whose proof is included as no proof as been found in the literature.
\begin{lemma}\label{L:cost}
Let $(X,\sfd,\mm)$ be a m.m.s. with $\diam (X) < D$. Let 
$f,\,g :X\to [0,\infty)$ be functions with $\|f \|_{L^{1}(X,\mm)} = \|g\|_{L^{1}(X,\mm)}$. Then 
$$
W_1(f \,\mm,g\,\mm)\leq D\|f-g\|_{L^1(X,\mm)}.
$$
\end{lemma}
\begin{proof}
Construct an admissible plan 
$\bar{\pi}\in \Pi(f\,\mm,g\,\mm)$, with $\bar{\pi}=\pi_1+\pi_2$ by defining 
\begin{align*}
\pi_1:=(Id,Id)_{\sharp}\left(g\,\mm\llcorner_{\lbrace g\leq f\rbrace}\right)+
(Id,Id)_{\sharp}\left(f\,\mm\llcorner_{\lbrace g>f\rbrace}\right)
\end{align*}
and considering any $\pi_2\in \Pi((f-g)^+\mm,(f-g)^-\mm)$. 
Then it is straightforward to check that
$$
W_1(f \,\mm,g\,\mm) \leq \int_{X\times X}\sfd(x,y)\,\pi_2(dxdy)\leq D \,\pi_2(X\times X)=D\int_{X}(f-g)^+\, \mm(dx),
$$
proving the claim.
\end{proof}

\begin{proposition}\label{P:upperboundW1}
Let $(X,\sfd,\mm)$ be a m.m.s. verifying $\RCD(K,N)$ 
and such that $\diam(X) = D < \infty$. 
Let 
$f$ be an eigenfunction of eigenvalue $\lambda> 2$. Then 
\begin{equation*}
    W_1(f^+\mm,f^-\mm)\leq C(K,N,D) 
    \sqrt{\frac{\log \lambda }{\lambda}}\|f\|_{L^1(X,\mm)},
\end{equation*}
with $C(K,N,D)$ growing linearly in $D$ and as square 
root in $N$.
\end{proposition}

\begin{proof}
We define 
$$
\mu_{0}^{\pm}:=f^{\pm}\,\mm,\qquad
\mu_t^{\pm}:=H_{t} \mu_{0}^{\pm},
$$
where $H_{t}$ is the heat flow (see Section \ref{Ss:Laplacian})
and by  triangular inequality
\begin{align*}
W_{1}(\mu_{0}^{+},\mu_{0}^{-})\leq W_{1}(\mu_{0}^{+},\mu_t^{+})+W_1(\mu_t^{+},\mu_t^{-})+W_1(\mu_t^{-},\mu_{0}^{-}),
\end{align*}
notice indeed that 
$\mu_{t}^{+}(X) =\mu_{0}^{+}(X)= \mu_{0}^{-}(X) = \mu_{t}^{-}(X)$.
Then by Theorem \ref{T:contraction} 
we deduce that  
\begin{align*}
W_{1}(\mu_{t}^{\pm},\mu_{0}^{\pm}) 
&~ = \left(\int_{X}f^{+}\,\mm\right) W_{1}(\mu_{t}^{\pm}/\mu_{t}^{\pm}(X),\mu_{0}^{\pm}/\mu_{0}^{\pm}(X))  \\
&~ \leq  \|f \|_{L^{1}(X,\mm)} W_{2}(\mu_{t}^{\pm}/\mu_{t}^{\pm}(X),\mu_{0}^{\pm}/\mu_{0}^{\pm}(X))\\
&~ \leq \sqrt{t}\|f \|_{L^{1}(X,\mm)}  C(t,K,N),
\end{align*}
where $C(t,K,N) : =\left(2N \frac{1 - e^{-K2t/3}}{K2t/3} \right)^{1/2}$, (with $C(t,K,N)\leq \sqrt{2N}$ if $K\geq 0$).

To bound $W_1(\mu_t^{+},\mu_t^{-})$ we use Lemma \ref{L:cost}. 
Call $g_{t}$ the evolution of a function $g$ through the heat flow 
($g_{t} = H_{t}g$), by the identification \eqref{E:identification}, 
it follows that (recall that $f \in W^{1,2}(X,\sfd,\mm)$ by definition)
$$
\mu_{t}^{\pm} = (H_{t} f^{\pm})\,\mm = f_{t}^{\pm} \,\mm.
$$
Notice that by infinitesimal Hilbertianity 
$$
f_{t}^{+} - f_{t}^{-} = H_{t}(f^{+} - f^{-}) = 
H_{t}(f) = e^{-\lambda t} f,
$$
where the last identity is a consequence of $f$ being an eigenfunction (see Section \ref{Ss:Laplacian}).
Then we have that 
$$
W_1(\mu_t^{+},\mu_t^{-})\leq D\|f^+_t-f^-_t\|_{L^1(X,\mm)}=
D \|f_t\|_{L^1(X,\mm)} = D e^{-\lambda t} \|f\|_{L^1(X,\mm)}.
$$
So finally 
$$
W_{1}(\mu_{0}^{+},\mu_{0}^{-})\leq 
\left(\sqrt{t} C(t,K,N)
+ D e^{-\lambda t}\right) \|f\|_{L^1(X,\mm)}.
$$
Choosing $t=\frac{1}{\lambda}\log(\lambda)$  we obtain 
$$
 W_1(f^+\,\mm,f^-\,\mm)\leq C(K,D,N) \sqrt{\frac{\log \lambda }{\lambda}}\|f\|_{L^1(X,\mm)},
$$
with $C(K,N,D)$ growing linearly in $D$  and as square 
root in $N$.
\end{proof}

Hence we can state one of the main results of this note.

\begin{theorem}[Nodal set $\RCD$-spaces]\label{T:nodalRCD}
Let $K,N \in \R$ with $N > 1$.
Let $(X,\sfd,\mm)$ be a m.m.s. satisfying $\RCD(K,N)$.  
Assume moreover $\diam(X) = D <\infty$.
Let $f$ be an eigenfunction of the Laplacian of eigenvalue 
$\lambda> 2$.
Then the following estimate is valid:
\begin{equation}\label{E:lowerboundRCD}
\Per\left(\{ x\in X\colon f(x)>0\}\right)
\geq\sqrt{ \frac{\lambda} {\log \lambda} }
\cdot 
\frac{\norm{f}_{L^1(X,\mm)}}{\bar C_{K,D,N} \|f \|_{L^{\infty}(X,\mm)}},
\end{equation}
where $\bar C_{K,D,N}$ grows linearly in $D$ if $K \geq 0$ and exponentially if $K< 0$ and grows with power $1/2$ in $N$.
\end{theorem}

\begin{proof}
Since $\diam(X) < \infty$, 
it follows that $\mm(X) < \infty$  and therefore 
$f \in L^{1}(X,\mm)$, it has zero mean and satisfies the growth conditions and regularity needed to invoke Theorem \ref{T:multidim}. 
Hence Theorem \ref{T:multidim} implies that 
$$
W_{1}(f^+\m,f^-\m)\cdot \Per\left(\{ x\in X\colon f(x)>0\}\right)
\geq \frac{\norm{f}^2_{L^1(X,\mm)}}{8C_{K,D}\|f \|_{L^{\infty}(X,\mm)}},
$$
that together with Proposition \ref{P:upperboundW1} 
implies that 
$$
\Per\left(\{ x\in X\colon f(x)>0\}\right) \geq 
\sqrt{\frac{\lambda}{\log \lambda}}
\frac{\norm{f}_{L^1(X,\mm)}}{C(K,N,D)C_{K,D}\|f \|_{L^{\infty}(X,\mm)}},
$$
giving therefore the claim.
\end{proof}

We are now in position of obtaining 
the explicit lower bound on the size of the nodal 
set of an eigenfunction stated in Theorem \ref{T:main5}.

\begin{proof}[Proof of Theorem \ref{T:main5}] 
It is a straightforward consequence of Theorem \ref{T:nodalRCD} and 
of the following observation: given an eigenfunction $f$ of eigenvalue $\lambda$, there exists a constant $C=C(K,N,D)$ such that 	
	\begin{equation*}
	\|f\|_{L^{\infty}(X,\mm)}\leq C \lambda^{\frac{N}{2}}\|f\|_{L^1(X,\mm)},
	\end{equation*}
	provided 	 $\lambda\geq D^{-2} $. 
	Indeed from \cite[Proposition 7.1]{AmbrHonPort} and assuming $\mm(X) = 1$, one has that 	
	\begin{equation*}
	\|f\|_{L^{\infty}(X,\mm)}\leq C \lambda^{\frac{N}{4}}\|f\|_{L^2(X,\mm)}\leq  C \lambda^{\frac{N}{4}}\|f\|^{\frac{1}{2}}_{L^{\infty}(X,\mm)}\|f\|^{\frac{1}{2}}_{L^1(X,\mm)}, 
	\end{equation*}
from which the claim follows dividing by the $L^{\infty}$ norm and squaring both sides. 
\end{proof}

\section{Linear combination of eigenfunctions}
\label{Ss:infinitesimalHilbertian}

We now consider functions obtained as linear 
combination of eigenfunctions. As expected, for the following 
results it will be necessary 
to assume the linearity of the Laplacian, 
i.e. infinitesimal Hilbertianity. 

We will however present two different upper bounds for the 
$W_{1}$ distance between the positive and the negative part of the function, one following the lines of Proposition \ref{P:upperboundW1} valid for $\RCD$ spaces
and one following Lemma \ref{L:nodalestimate} valid for 
$\MCP$ spaces.

\begin{proposition}\label{P:nodalcombination}
Let $(X,\sfd,\mm)$ be an essentially non-branching m.m.s. verifying 
$\MCP(K,N)$ with $\diam(X)=D < \infty$; assume moreover $(X,\sfd,\mm)$ to be infinitesimally Hilbertian. 

Let $f$ be a continuous function or, alternatively, $f \in W^{1,2}(X,\sfd,\mm)$, such that it satisfies in $L^2$ sense $f=\sum_{\lambda_k\geq\lambda}a_kf_{\lambda_k}$, $k\in\mathbb{N}$, where each $f_{\lambda_k}$ is an eigenfunction with eigenvalue 
$\lambda_{k}$.

Then the following estimate on the size of the nodal set of $f$ holds true:
$$
\Per(\{x\in X : f(x)>0 \})\geq   
\frac{\sqrt{\lambda} }{\sqrt{\mm(X)}C_{K,N,D}}  \cdot \frac{\|f\|^2_{L^1}}{\|f\|_{L^{2}}\|f\|_{L^{\infty}}},
$$
where $C_{K,N,D}$ is given by Theorem \ref{T:multidimMCP}.
\end{proposition}

\begin{proof}
From $\diam(X) < \infty$, 
it follows that $\mm(X) < \infty$  and therefore 
$f \in L^{1}(X,\mm)$, it has zero mean and satisfies the growth 
conditions needed to 
apply Theorem \ref{T:multidimMCP}. 
To prove the claim it will be therefore sufficient 
to obtain an upper bound for $W_{1}(f^{+}\,\mm,f^{-}\,\mm)$. 

Using the Kantorovich formulation, there exists a 
$1$-Lipschitz function such that
\begin{align*}
W_{1}(f^+\mm,f^-\mm)& = \int_{X}fu\,\mm 
=\sum_{\lambda_k\geq\lambda}a_k\int_{X}f_{\lambda_k}u\, \mm 
\leq\sum_{\lambda_k\geq\lambda}\frac{a_k}{\lambda_k} \||\nabla f_{\lambda_k}|_{w}\|_{L^2} 
\||\nabla u|\|_{L^2} \\ 
& \leq \sqrt{\mm(X)} \sum_{\lambda_k\geq\lambda}\frac{1}{\sqrt{\lambda_k}}\|{a_k} f_{\lambda_k}\|_{L^2}
\leq \frac{\sqrt{\mm(X)}}{\sqrt{\lambda}}\sum_{\lambda_k\geq\lambda}\|a_k f_{\lambda_k}\|_{L^2}
=\frac{\sqrt{\mm(X)}}{\sqrt{\lambda}}\|f\|_{L^2},
\end{align*}
where we used in the third identity  
$\|\frac{1}{\lambda_k}|\nabla f_{\lambda_k}|_{w}\|^2_{L^2}
=\frac{1}{\lambda_k}\|f_{\lambda_k}\|^2_{L^2}$, 
and in the last one the orthogonality of $\lbrace f_{\lambda_k}\rbrace_{k \in \N}$ given by infinitesimally Hilbertianity.
\end{proof}

\begin{lemma}\label{L:upperboundlinearcombination}
Let $(X,\sfd,\mm)$ be a m.m.s. verifying $\RCD(K,N)$ 
and such that $\diam(X) = D < \infty$ and $K \geq 0$.
Let $f:X\to \Real$ be a continuous or, alternatively, $f \in W^{1,2}(X,\sfd,\mm)$, such that
$$
f=\sum_{\lambda_k\geq \lambda}\langle f,f_{\lambda_k}\rangle f_{\lambda_{k}}, \qquad \{\lambda_{k}\}_{k\in \N}, 
$$
where $\lbrace f_{\lambda_k}\rbrace_{k\in \N}$ are 
 eigenfunctions of the Laplacian 
of unitary $L^2$-norm with eigenvalue $\lambda_{k}$, 
$\langle f,f_{\lambda_k}\rangle$ is the scalar product of 
$L^{2}(X,\mm)$, $\langle f_{\lambda_{j}},f_{\lambda_k}\rangle = \delta_{j,k}$ and 
convergence of the series is in $L^{2}(X,\mm)$.
Then
$$
W_1(f^+\mm,f^-\mm)\leq C(K,N,D,\mm(X)) \left(\frac{1}{\lambda}\log\left({\lambda} \frac{\|f\|_{L^2}}{\|f\|_{L^{1}}}\right)\right)^{\frac{1}{2}}\|f\|_{L^1},
$$
with $C(K,N,D,\mm(X))$ an explicit constant, provided that $\lambda\geq 2\sqrt{\mm(X)}$.
\end{lemma}

\begin{proof}
Following the approach and the same notation of the proof of Proposition \ref{P:upperboundW1} we have
$$
W_{1}(\mu_{0}^{+},\mu_{0}^{-})\leq W_{1}(\mu_{0}^{+},\mu_t^{+})+W_1(\mu_t^{+},\mu_t^{-})+W_1(\mu_t^{-},\mu_{0}^{-}),
$$
and deduce from Theorem \ref{T:contraction} that  
$$
W_{1}(\mu_{t}^{\pm},\mu_{0}^{\pm}) 
 \leq \sqrt{t}\|f \|_{L^{1}(X,\mm)}  C(t,K,N),
$$
where $C(t,K,N) : =\left(2N \frac{1 - e^{-K2t/3}}{K2t/3} \right)^{1/2}$.
Then to bound $W_1(\mu_t^{+},\mu_t^{-})$, again using Lemma \ref{L:cost}, 
by orthonormality of $\{ f_{\lambda_{k}}\}_k$ it follows that
\begin{align}\label{E:cs}
\|f_{t}\|^2_{L^1(X,\mm)}
=&~\left\|\sum_{\lambda_k\geq \lambda}e^{-\lambda_k t}\langle f,f_{\lambda_k}\rangle f_{\lambda_k}\right\|^2_{L^1(X,\mm)} 
\leq  \mm(X) \left\|\sum_{\lambda_k\geq \lambda}e^{-\lambda_k t}\langle f,f_{\lambda_k}\rangle f_{\lambda_k}\right\|_{L^2(X,\mm)}^2  \nonumber\\
=&~ \mm(X) \sum_{\lambda_k\geq \lambda}e^{-2\lambda_k t}|\langle f,f_{\lambda_k}\rangle |^2
\leq \mm(X) e^{-2\lambda t}\|f\|^2_{L^2(X,\mm)}.
\end{align}
So finally 
$$
W_{1}(\mu_{0}^{+},\mu_{0}^{-})\leq 
\sqrt{t}\|f \|_{L^{1}(X,\mm)}  C(t,K,N) +
D\sqrt{\mm(X)} e^{-\lambda t}\|f\|_{L^2(X,\mm)}.
$$
Using that $K \geq 0$ (so $C(t,K,N)\leq \sqrt{2N}$) and choosing 
$t=\frac{1}{\lambda}\log\left(\frac{\lambda\|f\|_{L^2(X,\mm)}}{\|f\|_{L^1(X,\mm)}}\right)$, 
it holds
\begin{equation*}
W_{1}(\mu_{0}^{+},\mu_{0}^{-})\leq C(K,N,D, \mm(X)) \left(\frac{1}{\lambda}\log\left(\lambda\frac{\|f\|_{L^2}}{\|f\|_{L^{1}}}\right)\right)^{\frac{1}{2}}\|f\|_{L^1},
\end{equation*}
proving the claim.
\end{proof}

The following result is then a straightforward consequence

\begin{corollary}
Let $(X,\sfd,\mm)$ be a m.m.s. verifying $\RCD(K,N)$ 
and such that $\diam(X) = D < \infty$. 

Let $f:X\to \Real$ be a continuous or, alternatively, $f \in W^{1,2}(X,\sfd,\mm)$ such that
$$
f=\sum_{\lambda_k\geq \lambda}\langle f,f_{\lambda_k}\rangle f_{\lambda_{k}}, \qquad \{\lambda_{k}\}_{k\in \N}, \ \lambda >0,
$$
where $\lbrace f_{\lambda_k}\rbrace_{k\in \N}$ are 
 eigenfunctions of the Laplacian 
of unitary $L^2$-norm with eigenvalue $\lambda_{k}$, 
$\langle f,f_{\lambda_k}\rangle$ is the scalar product of 
$L^{2}(X,\mm)$, $\langle f_{\lambda_{j}},f_{\lambda_k}\rangle = \delta_{j,k}$ and 
convergence of the series is in $L^{2}(X,\mm)$.

Then the following estimate on the size of the nodal set of $f$ holds true:
$$
\Per(\{x\in X : f(x)>0 \})\geq   
\frac{\sqrt{\lambda} }{C(K,N,D,\mm(X))} 
\log\left(\lambda \frac{\|f\|_{L^2}}{\|f\|_{L^{1}}}\right)^{-1/2} \cdot 
\frac{\|f\|_{L^1}}{\|f\|_{L^{\infty}}},
$$
with $C(K,N,D,\mm(X))$ the same constant of Lemma \ref{L:upperboundlinearcombination}, provided that {$\lambda\geq 2\sqrt{\mm(X)}$}.
\end{corollary}


\end{document}